\documentclass[12pt]{amsart}

\usepackage[T1]{fontenc}
\usepackage{lmodern}
\usepackage{microtype}
\usepackage{amsmath,mathtools}

\usepackage[english]{babel}

\usepackage{vmargin,graphicx,amssymb,color,subfigure,enumerate,csquotes, dsfont, mathrsfs}

\numberwithin{equation}{section}

\usepackage{color}
\definecolor{gr}{rgb}   {0.,   0.69,   0.23 }
\definecolor{bl}{rgb}   {0.,   0.5,   1. }
\definecolor{mg}{rgb}   {0.85,  0.,    0.85}
\definecolor{or}{rgb}   {0.9,  0.5,   0.}

\definecolor{webred}{rgb}{0.75,0,0}
\definecolor{webgreen}{rgb}{0,0.75,0}
\usepackage[citecolor=webgreen,colorlinks=true,linkcolor=webred]{hyperref}

\newtheorem{theorem}{Theorem}[section]
\newtheorem{lemma}[theorem]{Lemma}

\newtheorem{proposition}[theorem]{Proposition}
\newtheorem{corollary}[theorem]{Corollary}
\theoremstyle{definition}
\newtheorem{definition}[theorem]{Definition\rm}
\newtheorem{assumption}[theorem]{Assumption}
\theoremstyle{remark}
\newtheorem{remark}[theorem]{Remark}

\newcommand{\dist}{\mathsf{dist}}
\newcommand{\Tr}{\mathsf{Tr}}
\renewcommand{\Re}{\mathsf{Re}}
\newcommand{\loc}{\mathsf{loc}}
\newcommand{\e}{\mathsf{e}}
\newcommand{\x}{\mathbf{x}}
\newcommand{\y}{\mathbf{y}}
\newcommand{\n}{\mathbf{n}}
\newcommand{\Dom}{\mathsf{Dom}}
\newcommand{\supp}{\mathsf{supp}\,}

\newcommand{\h}{\hbar}

\newcommand{\Z}{\mathbb{Z}}
\newcommand{\R}{\mathbb{R}}
\newcommand{\C}{\mathbb{C}}
\newcommand{\A}{\mathbf{A}}
\newcommand{\kb}{\mathbf{k}}
\newcommand{\cb}{\mathbf{c}}
\newcommand{\zb}{\mathbf{0}}
\newcommand{\B}{\mathbf{B}}
\newcommand{\Dir}{\mathbf{Dir}}
\newcommand{\Neu}{\mathbf{Neu}}
\newcommand{\sH}{\mathsf{H}}
\newcommand{\sL}{\mathsf{L}}
\newcommand{\Id}{\mathsf{Id}}
\newcommand{\eps}{\varepsilon}
\newcommand{\dx}{\,\mathrm{d}}

\begin{document}

\title[]{Semiclassical Sobolev constants for the electro-magnetic Robin Laplacian}

\author{S. Fournais}
\address[S. Fournais]{Aarhus University, Ny Munkegade 118, DK-8000 Aarhus~C, Denmark}
\email{fournais@math.au.dk}
\author{L. Le Treust}
\address[L. Le Treust]{IRMAR, Universit\'e de Rennes 1, Campus de Beaulieu, F-35042 Rennes cedex, France}
\email{loic.letreust@univ-rennes1.fr}
\author{N. Raymond}
\address[N. Raymond]{IRMAR, Universit\'e de Rennes 1, Campus de Beaulieu, F-35042 Rennes cedex, France}
\email{nicolas.raymond@univ-rennes1.fr}
\author{J. Van Schaftingen}
\address[J. Van Schaftingen]{Universit\'e catholique de Louvain,
Institut de Recherche en Mathématique et Physique,
Chemin du Cyclotron 2 bte L7.01.01, B-1348 Louvain-la-Neuve, Belgium}
\email{jean.vanschaftingen@uclouvain.be }
\date{\today}

\maketitle
\begin{abstract}
This paper is devoted to the asymptotic analysis of the optimal Sobolev constants in the semiclassical limit and in any dimension. We combine semiclassical arguments and concentration-compactness estimates to tackle the case when an electro-magnetic field is added as well as a smooth boundary carrying a Robin condition. As a byproduct of the semiclassical strategy, we also get exponentially weighted localization estimates of the minimizers.
\end{abstract}
\tableofcontents
\section{Introduction}

\subsection{Description of the problem}
The aim of this work is to investigate optimal Sobolev constants under an electro-magnetic field and in a domain with a smooth boundary. We want especially to investigate the behavior of these constants in the semiclassical limit.

\subsubsection{Geometric context}
Before describing our main results, we describe the geometric context of this paper.

For $d\geq 2$, we consider an open, bounded and simply connected set $\Omega\subset\R^d$ with smooth boundary. We also introduce the smooth electro-magnetic potential $(V,\A)\in\mathcal{C}^\infty(\overline{\Omega},\R\times\R^d)$ and the variable Robin coefficient $\gamma\in\mathcal{C}^\infty(\partial\Omega,\R)$. We let $\mathcal{G}=(\Omega,\Id,  V,\A,\gamma)$ where $\Id$ stands for the standard Euclidean metric. 

\begin{definition}
For notational convenience, we will constantly consider \emph{quintuples gathering the Robin electro-magnetic geometry} $\mathsf{G}=(U,\mathsf{R}, \mathsf{V},\mathsf{A},\mathsf{c})$ where 
\begin{enumerate}[i.]
\item $U$ is a smooth open set,
\item $\mathsf{R}$ is a Riemannian metric on $\overline{U}$,
\item the electric potential $\mathsf{V}$ belongs to $\mathcal{C}^\infty(\overline{U},\R)$,
\item the magnetic vector potential $\mathsf{A}$ belongs to $\mathcal{C}^\infty(\overline{U},\R^d)$,
\item the Robin coefficient $\mathsf{c}$ belongs to $\mathcal{C}^\infty(\partial U,\R)$. 
\end{enumerate}
If $\Phi : U'\to U$ is a  local chart near the boundary, then we introduce the pull-back geometry
\[\Phi^*\mathsf{G}=(U', (d\Phi)^\mathsf{T}\mathsf{R}(d\Phi), \mathsf{V}\circ\Phi, \left(d\Phi\right)^{\mathsf{T}}\circ\mathsf{A}\circ \Phi,\gamma\circ\Phi)\,.\]
\end{definition}
We recall that the \enquote{magnetic field} is the $2$-form defined as the exterior derivative
\[\mathsf{B}=\dx\mathsf{A}=\dx \left(\sum_{j=1}^d \mathsf{A}_{j}\dx x_{j}\right)\,,\]
where $\mathsf{A}$ is identified with a $1$-form thanks to the Euclidean duality. The $2$-form $\mathsf{B}$ may be identified with the skew-symmetric matrix, called \enquote{magnetic matrix}, $(\mathsf{B}_{k\ell})_{1\leq k,\ell\leq d}$ where $\mathsf{B}_{k\ell}=\partial_{k}\mathsf{A}_{\ell}-\partial_{\ell}\mathsf{A}_{k}$. 
It is well known that the non-zero eigenvalues of the matrix $\mathsf{B}$ are in the form $(\pm i\beta_{k})_{1\leq k\leq \lfloor\frac{d}{2}\rfloor}$, $\beta_{k} > 0$ and that $0$ is always an eigenvalue in odd dimension. This allows to define
\[\Tr^+\,\mathsf{B}=\sum_{k=1}^{\lfloor\frac{d}{2}\rfloor}\beta_{k}\,.\]
In particular, if $\Tr^+\,\mathsf{B}=0$, then $\mathsf{B}=0$.

\begin{definition}
We will say that the geometry $\mathsf{G}$ is \emph{homogeneous} when $(\mathsf{V}, \mathsf{B}, \mathsf{c})$ is constant and when $U$ is the whole space or a half-space, equipped with the Euclidean metric $\mathsf{R}=\Id$. We will also say that a geometry is Euclidean when $\mathsf{R}=\mathsf{Id}$. In this case, we will also use the notation $\underline{\mathsf{G}}=(\R^d,\Id, \mathsf{V},\mathsf{A}, 0)$, where $\mathsf{A}$ is a linear potential associated with $\mathsf{B}$.
\end{definition}

\subsubsection{Minimization problem}

We now introduce the minimization problem under consideration in this paper.

Let $p\in\left[2,2^*\right)$, with $2^*=\frac{2d}{d-2}$. We are mainly interested in the following \enquote{optimal Sobolev constant}, in the case of a Euclidean geometry \(\mathcal{G}\),
\begin{equation}\label{eq.Sobo}
\lambda(\mathsf{G}, h, p)=\inf_{\underset{\psi\neq 0}{\psi\in \sH_{\A}^1(U),}}\frac{\mathfrak{Q}_{\mathsf{G},h}(\psi)}{\|\psi\|^2_{\sL^p(U)}}\,,
\end{equation}
where the magnetic Sobolev space is defined by 
\[
  \sH^1_{\A}(U) = \{\psi \in \sL^2(U) : (-ih\nabla + \A)u\in \sL^2(U)\}
\]
and for all $\psi\in\sH^1_{\A}(U)$, the quadratic form \(\mathfrak{Q}_{\mathsf{G},h}\) is defined by
\begin{equation}\label{eq.fq}
  \mathfrak{Q}_{\mathsf{G},h}(\psi)=\int_{U}|(-ih\nabla+\mathsf{A})\psi|^2+h\mathsf{V}|\psi|^2\dx\x+h^{\frac{3}{2}}\int_{\partial U} \mathsf{c} |\psi|^2\dx \sigma(\x)\,.
\end{equation}
Here, $\dx\sigma$ is the surface measure on the boundary $\partial U$. 

\subsubsection{Homogeneity}
Let us heuristically explain where the different powers of $h$ come from. Let us introduce the temporary semiclassical parameter $\h$. We consider the initial quadratic form:
\begin{equation}\label{eq.initial}
\int_{\Omega}|(-i\h^a\nabla+\h^b\A)\psi|^2+\h^c V|\psi|^2\dx\x+\h^{d}\int_{\partial \Omega} \gamma |\psi|^2\dx \sigma(\x)\,.
\end{equation}
After a semiclassical local zoom, we would like to get an homogeneous quadratic form. It is sufficient to derive these appropriate powers \enquote{locally}, that is in the case of a homogeneous geometry (and $\Omega$ being replaced for instance by the half-space $U$):
\begin{equation}\label{eq.initial'}
\int_{U}|(-i\h^a\nabla+\h^b\mathsf{A})\psi|^2+\h^c \mathsf{V}|\psi|^2\dx\x+\h^{d}\int_{\partial U} \mathsf{c} |\psi|^2\dx \sigma(\x)\,.
\end{equation}
Let us determine the $a, b, c, d$ that lead to non-trivial situations.

First, we may always reduce the investigation to $a=1$ by multiplying the quadratic form by an appropriate power of $\hbar$. Then, we would like that, up to a semiclassical zoom, all the different quantities play on the same scale (if not, this would mean that an effect could be neglected). Thus, we let $\x=\h^{\eta}\y$ with $\eta\neq 0$ and consider the rescaled quadratic form
\[
\int_{U}|(-i\h^{1-\eta}\nabla+\h^{b+\eta}\mathsf{A})\psi|^2+\h^{c}\mathsf{V}|\psi|^2\dx\y+\h^{d-\eta}\int_{\partial U} \mathsf{c} |\psi|^2\dx \sigma(\y)\,.
\]
In order to balance all the electro-magnetic effects, we choose 
\(c = 2 - 2 \eta =2 b + 2 \eta = d-\eta\). We get
\[b=c-1\,,\quad d=1+\frac{c}{2}\,,\quad \eta=1-\frac{c}{2}\,.\]
Note that $\eta\neq 0$ means that $c\neq 2$ and that $c=2$ corresponds then to a homogeneous problem (which is not semiclassical!). 

Therefore, coming back to \eqref{eq.initial}, this leads to
\[
\int_{\Omega}|(-i\h\nabla+\h^{c-1}\A)\psi|^2+\h^{c}V|\psi|^2\dx\x+\h^{1+\frac{c}{2}}\int_{\partial \Omega} \gamma |\psi|^2\dx \sigma(\x)\,.
\]
Now, if $c-1>1$, this quadratic form is locally a perturbation of the one of $-\h^2\Delta$ (with Neumann condition) and thus the Robin-electro-magnetic geometry can be forgotten. Thus, if we are interested in geometric effects, we only have to consider $c<2$. In this case, we write
\[
\h^{2-2c}\left\{\int_{\Omega}|(-i\h^{2-c}\nabla+\A)\psi|^2+\h^{2-c}V|\psi|^2\dx\x+\h^{\frac{3}{2}(2-c)}\int_{\partial \Omega} \gamma |\psi|^2\dx \sigma(\x)\right\}\,,
\]
and we can consider $h=\h^{2-c}$ as new semiclassical parameter. We get the powers appearing in \eqref{eq.fq}.

\subsubsection{Basic properties}
We can already make some elementary observations that we will constantly use.

We first recall the diamagnetic inequality (see for example \cite[Theorem 7.21]{LiebLoss2001}, \cite[Theorem 2.1.1]{FouHel10}):
\[\forall \psi\in\sH^1_{\A}(U)\,,\qquad\|\nabla|\psi|\|^2_{\sL^2(U)}\leq \|(-i\nabla+\A)\psi\|^2_{\sL^2(U)}\,.\]
This inequality implies that $|\psi|\in \sH^1(U)$ and we get, thanks to the classical trace theorem, that its trace is well-defined as an element of $\sH^{\frac{1}{2}}(\partial U)$; thus $\mathfrak{Q}_{\mathsf{G},h}$ is well-defined on $\sH^1_{\A}(U)$. 
Another important property of the magnetic Laplacian is the gauge invariance (see for example \cite[\S 7.21]{LiebLoss2001}): 
\begin{equation}\label{eq.gauge}
\forall\varphi\in\mathcal{C}^\infty(\overline{U}),\qquad\mathfrak{Q}_{\mathsf{G},h}(e^{i\varphi/h}\psi)=\mathfrak{Q}_{\mathsf{G}^{\varphi},h}(\psi),\quad \text{ with }\quad  \mathsf{G}^\varphi=(U,\mathsf{R},\mathsf{V},\mathsf{A}+\nabla\varphi,\mathsf{c})\,.
\end{equation}
Let us already notice that it is not clear whether the infimum \eqref{eq.Sobo} actually exists when $U$ is unbounded. Nevertheless, if $\mathsf{V}$ and $\mathsf{c}$ are non-negative, its existence is obvious. In any case, when this infimum exists and is a minimum, the corresponding minimizers satisfy, in the sense of distributions, the following nonlinear focusing equation
\begin{equation}\left\{
\begin{aligned}
(-ih\nabla+\mathsf{A})^2\psi+h\mathsf{V}\psi&=\lambda(\mathsf{G},h,p)|\psi|^{p-2}\psi\,,\\ 
(-ih\nabla+\mathsf{A})\psi\cdot\n&=-ih^{\frac{1}{2}}\mathsf{c}\psi,\text{ on }\partial U\,,
\end{aligned}\right.
\end{equation}
where we assumed that $\|\psi\|_{\sL^p(U)}=1$ and where $\n$ is the \emph{inward} unit normal to the boundary. By multiplying $\psi$ by an appropriate constant, we therefore have a  solution (for $p>2$) of the following stationary Schr\"odinger nonlinear equation
\begin{equation}\label{eq.SNLS}\left\{
\begin{aligned}
(-ih\nabla+\mathsf{A})^2\Psi+h\mathsf{V}\Psi&=|\Psi|^{p-2}\Psi\,,\\ 
(-ih\nabla+\mathsf{A})\Psi\cdot\n&=-ih^{\frac{1}{2}}\mathsf{c}\Psi,\text{ on }\partial U\,.
\end{aligned}\right.
\end{equation}
As a byproduct of our investigation, we will get the existence of non-trivial solutions of \eqref{eq.SNLS} (solitons) that are localized (in the semiclassical limit) near the minima of a \emph{concentration function} describing the local nonlinear electro-magnetic Robin geometry.

\subsubsection{Mathematical context and motivations}
The aim of this paper is to estimate the optimal Sobolev constant $\lambda(\mathcal{G},h,p)$ under generic assumptions on the geometry. 

In the linear case, \emph{i.e.}\  when $p=2$, this problem has now a long history, especially in two and three dimensions in the case of Neumann boundary conditions and $V=0$. The investigation of the lowest eigenvalue of the semiclassical magnetic Laplacian can be motivated by the theory of superconductivity and the study of the third critical field in the Ginzburg-Landau theory. The reader may consult the book by Fournais and Helffer \cite{FouHel10} or the one by Raymond \cite{Ray14} for an introduction to these topics. In this linear and purely magnetic framework, it appears that the microlocalization of the eigenfunctions is strongly related to the asymptotics of the lowest eigenvalue. This fact was noticed, for instance, in the papers by Helffer and Morame \cite{HelMo01, HelMo04} where numerous techniques have been developed to analyze the magnetic Laplacian and its eigenfunctions. Even more recently in \cite{HelKo15, RVN15, HKRVN14}, in cases without boundary, subtle localization properties of the magnetic eigenfunctions have played a fundamental role in the semiclassical spectral theory (and we will meet again this aspect in the nonlinear context). In cases with boundaries, the Robin condition is  physically motivated by inhomogeneous superconductors (see for instance the linear and nonlinear contributions by Kachmar \cite{K06a, K06b, K08, K15}): in this context, the Robin condition is sometimes called \enquote{de Gennes condition}. In the linear framework many recent contributions have also been made to investigate the semiclassical curvature effects with Robin condition (see for instance \cite{EMP, PP-eh, KKR16} and also \cite{HelKacRay15} in the case with symmetries).

In the nonlinear case $p>2$, the theory does not seem as developed as in the linear case, especially when a magnetic field and a boundary are added. In the seminal paper \cite{EL89} and in the concentration-compactness spirit, it is proved that $\lambda(\mathsf{G}, 1, p)$ is attained when $\mathsf{G}=(\R^d,\mathsf{Id}, 0,\mathsf{A},0)$, when $\mathsf{B}$ is constant and non-zero and when $p$ is subcritical. In \cite{dCvS15}, the authors have analyzed the semiclassical situation and obtained, up to subsequence extraction of the semiclassical parameter, the one term asymptotics of $\lambda(\mathcal{G}, h, p)$ with the geometry $\mathcal{G}=(\Omega, \mathsf{Id}, V, \A, +\infty)$, when $\Omega$ bounded and $\Tr^+\,\B+V$ does not vanish. The idea in \cite{dCvS15} was to use a semiclassical blow up argument near each point $\x\in\Omega$ and compare with nonlinear models with constant electro-magnetic field $(V_{\x}, \B_{\x})$. In particular the minimizers are essentially localized near the minima of the \emph{concentration function} $\Omega\ni\x\mapsto\lambda((\R^d,\mathsf{Id}, V_{\x}, \A_{\x},0), 1, p)$, where $\A_{\x}$ is a linear potential associated with the constant field $\B_{\x}$ (see also \cite{BNvS} where some properties of the concentration function are discussed). As we mentioned above, the localization properties of the magnetic eigenfunctions are strongly connected to the eigenvalue asymptotics and this phenomenon is expected to be even stronger in the nonlinear framework. 

The present paper aims at extending the theory developed in \cite{FR15} (in two dimensions without boundary) by investigating the effect of a smooth boundary carrying a Robin condition, in any dimension. For that purpose, we will decouple the semiclassical linear methods (described in \cite[Part I]{FouHel10}) and the concentration-compactness arguments. By doing so we will derive a quantitative remainder in the semiclassical asymptotics of $\lambda(\mathcal{G}, h, p)$ as well as quantitative localization estimates of the minimizers.

\subsection{Assumptions and main results}
We can now state our main assumptions and results. Let us first explain in which framework our problem is well-posed.
\begin{lemma}
The quadratic form $\mathfrak{Q}_{\mathcal{G},h}$ is bounded from below and defines a self-adjoint operator $\mathfrak{L}_{\mathcal{G},h}$ with compact resolvent whose domain is
\begin{multline*}
\Dom\left(\mathfrak{L}_{\mathcal{G},h}\right)=\big\{\psi\in\sH^1(\Omega) : ((-ih\nabla+\A)^2+hV)\psi\in\sL^2(\Omega)\\
\text{ and } (-ih\nabla+\A)\psi\cdot\n(\x)=-ih^{\frac{1}{2}}\gamma(\x)\psi(\x),\quad\x\in\partial\Omega\big\}\,.
\end{multline*}
In particular, $\lambda(\mathcal{G}, h, 2)$ coincides with its lowest eigenvalue.
\end{lemma}
\begin{remark}
We recall that $\Omega$ is bounded and that $V$ and $\A$ are smooth on $\overline{\Omega}$. Therefore, if $\psi\in\sH^1(\Omega)$ and $((-ih\nabla+\A)^2+hV)\psi\in\sL^2(\Omega)$, then $\psi\in\sH^2(\Omega)$ so that the Robin boundary condition is well-defined by a classical trace theorem.
\end{remark}
We will provide a sufficient condition, for the geometry $\mathcal{G}$, that ensures that the $\sL^2$ norm is controlled by $\mathfrak{Q}_{\mathcal{G}, h}$ in the semiclassical limit $h\to 0$. This condition will be related to models with homogeneous geometry. Let us recall that, for all $\x_{0}\in\overline{\Omega}$, the vector potential, defined in a convex neighborhood of $\x_{0}$,
\begin{equation}\label{eq.Lorentz}
\langle\A^{\mathsf{L}}_{\x_{0}}(\x),\cdot\rangle_{\R^d}=\int_{0}^1t \B_{\x_{0}+t(\x-\x_{0})}(\x-\x_{0},\cdot)\dx t
\end{equation}
satisfies, in this neighborhood, 
\begin{equation}\label{eq.jaugex0}
\A^{\mathsf{L}}_{\x_{0}}(\x_{0})=0\,\qquad\text{ and }\qquad\dx \A^\mathsf{L}_{\x_{0}}=\B\,. 
 \end{equation}
 We introduce its linear approximation
\begin{equation}\label{eq.linear-approx}
\mathcal{A}^\mathsf{L}_{\x_{0}}(\x)=\frac{1}{2}\B(\x_{0})(\x-\x_{0})\,.
\end{equation}
We will meet the following homogeneous Euclidean geometries:
\begin{enumerate}[i.]
\item if $\x_{0}\in\Omega$, we consider $\mathcal{G}_{\x_{0}}=(\x_{0}+\R^d, \mathsf{Id}, V(\x_{0}), \mathcal{A}^\mathsf{L}_{\x_{0}}, 0)$,
\item if $\x_{0}\in\partial\Omega$, we consider $\mathcal{G}_{\x_{0}}=(\x_{0}+\mathsf{T}_{\x_{0}}(\partial\Omega)+\R_{+}\n(\x_{0}),\mathsf{Id}, V(\x_{0}), \mathcal{A}^\mathsf{L}_{\x_{0}}, \gamma(\x_{0}))$, where $\mathsf{T}_{\x_{0}}(\partial\Omega)$ is the linear tangent space of $\partial\Omega$ at $\x_{0}$.
\end{enumerate}

Let us now state our main assumption which is of spectral nature: we assume that the $2$-eigenvalue is not degenerate.
\begin{assumption}\label{a.0}
We assume that 
\begin{enumerate}[i.]
\item\label{a.0i} $\Omega\ni\x\mapsto \lambda(\mathcal{G}_{\x},1,2)=\Tr^+\,\B(\x)+V(\x)$ does not vanish,
\item\label{a.0ii} $\partial\Omega\ni\x\mapsto\lambda(\mathcal{G}_{\x},1,2)$ is bounded from below by a positive constant.
\end{enumerate}
\end{assumption}
We will provide sufficient conditions under which Assumption~\ref{a.0} is satisfied in Section~\ref{sec.suff}. Before presenting our main result, let us state a proposition that ensures that the infimum \eqref{eq.Sobo} is actually well-defined and a minimum.
\begin{proposition}\label{prop.minorationL2}
There exist $h_{0}, C>0$ such that, for all $h\in(0,h_{0})$, we have
\[\lambda(\mathcal{G}, h, 2)\geq h\inf_{\x\in\overline{\Omega}}\lambda(\mathcal{G}_{\x},1,2)-Ch^{\frac{5}{4}}> 0\,,\]
and, under Assumption~\ref{a.0}, the infimum \eqref{eq.Sobo} for $\mathsf{G}=\mathcal{G}$ is a minimum.
\end{proposition}

We can transform Assumption~\ref{a.0} (related to the positivity of the spectrum) into a semi-continuity property of the $p$-eigenvalue which will play a crucial role in our investigation. This semi-continuity will be derived from a concentration-compactness analysis and used when estimating the Sobolev constants from above.
\begin{proposition}\label{prop.semicont}
Under Assumption~\ref{a.0}, the function $\x\mapsto \lambda(\mathcal{G}_{\x}, 1, p)$ is lower semi-continuous on $\overline{\Omega}$ for $p\in\left(2, 2^*\right)$.
\end{proposition}
Our main theorem is the following accurate estimate of the optimal Sobolev constant with electro-magnetic field and Robin condition on the boundary, in the semiclassical limit.
\begin{theorem}\label{theo.1}
Let $p\in\left(2, 2^*\right)$. Under Assumption~\ref{a.0}, there exist $h_{0}>0, C>0$ such that, for all $h\in(0,h_{0})$,
\[  h^{\frac{d}{2}-\frac{d}{p}}h(1-Ch^{\frac{1}{6}})\inf_{\x\in\overline{\Omega}}\lambda(\mathcal{G}_{\x},1,p)\leq \lambda(\mathcal{G}, h, p)\leq  h^{\frac{d}{2}-\frac{d}{p}}h(1+Ch^{\frac{1}{2}}|\log h|)\inf_{\x\in\overline{\Omega}}\lambda(\mathcal{G}_{\x},1,p)\,. \]
In the case when there exists $\x_{0}\in\partial\Omega$ such that
  \[
    \inf_{\x\in\overline{\Omega}}\lambda(\mathcal{G}_{\x},1,p) = \lambda(\mathcal{G}_{\x_0},1,p)<\lambda(\underline{\mathcal{G}}_{\x_0},1,p),
  \]
the logarithm appearing in the upper bound can be removed.
\end{theorem}
  
By Proposition~\ref{prop.semicont}, we may consider the set $\mathcal{M}\subset\overline{\Omega}$ of the minimizers of the concentration function $\x\mapsto\lambda(\mathcal{G}_{\x}, 1, p)$. In relation with the estimate of Theorem~\ref{theo.1}, we can deduce the following (exponential) decay estimate of the minimizers away from $\mathcal{M}$.
\begin{theorem}\label{theo.AgmonNLS}
Let $p\in\left(2, 2^*\right)$. Under Assumption~\ref{a.0}, for all $\eps>0$ we define 
\begin{equation}\label{eq.Me}
\mathcal{M}_{\eps}=\mathcal{M}+D(0,\eps)\,.
\end{equation}
Then, for all $\eps>0$ and $\rho\in (0,\frac{1}{2} )$, there exist $h_{0}>0, C>0$ such that, for all $h\in(0,h_{0})$ and all $\sL^p$-normalized minimizers $\psi_{h}$,
\[\|\psi_{h}\|_{\sL^p(\complement \mathcal{M}_{\eps})}\leq Ce^{-\eps h^{-\rho}}\,.\]
\end{theorem}

\subsection{Further results}
Let us now describe two applications or extensions of our results and methods.
\subsubsection{Large smooth domains}
Let us consider a smooth domain $\Omega\subset\R^d$ and, for $R\geq 1$, the dilated domain $\Omega_{R}=R\,\Omega$. In this section, we consider $V=1$, $\A=0$, and $\gamma=0$. For $p\in[2,2^*)$, we introduce the classical Sobolev constant
\[
  \lambda^\Neu(\Omega_{R}, p)
    =\lambda ((\Omega_{R}, \mathsf{Id}, 1, 0, 0), 1, p)
    =\inf_{\underset{\psi\neq 0}{\psi\in \sH^1(\Omega_{R}),}}\frac{\int_{\Omega_{R}}|\nabla\psi|^2+|\psi|^2\dx\x}{\|\psi\|^2_{\sL^p(\Omega_{R})}}\,.
\]
We have the semiclassical reformulation
\[\lambda^\Neu(\Omega_{R}, p)=R^{4-\frac{4}{p}}\inf_{\underset{\psi\neq 0}{\psi\in \sH^1(\Omega),}}\frac{\int_{\Omega}h^2|\nabla\psi|^2+h|\psi|^2\dx\x}{\|\psi\|^2_{\sL^p(\Omega)}}=R^{4-\frac{4}{p}}\lambda(\Omega, h, p)\,,\qquad h=R^{-2}\,.\]
Note that, by a symmetrization argument,
\[\lambda^\Neu(\R^{d}_{+}, p)\leq\left(\frac{1}{2}\right)^{1-\frac{2}{p}}\lambda^\Neu(\R^{d}, p)\,.\]
Thus we may directly apply our result.
\begin{corollary}\label{cor.large.dom}
Let $p\in\left(2, 2^*\right)$. There exist $C, R_{0}>0$ such that, for all $R\geq R_{0}$, 
\[
  R^{2-d+\frac{2d-4}{p}} (1-CR^{-\frac{1}{3}})\lambda^\Neu(\R^{d}_{+}, p)
  \leq\lambda^\Neu(\Omega_{R}, p)
  \leq R^{2-d+\frac{2d-4}{p}}(1+CR^{-1})\lambda^\Neu(\R^{d}_{+}, p)\,.
\]
Moreover, for all $\eps>0$ and $\rho\in\left(0,\frac{1}{2}\right)$, there exist $R_{0}>0, C>0$ such that, for all $R\geq R_{0}$ and all associated $\sL^p$-normalized minimizers $\psi$,
\[\|\psi\|_{\sL^p(\complement \mathcal{M}_{\eps})}\leq Ce^{-\eps R^{\rho}}\,,\]
where $\mathcal{M}_{\eps}$ is an $\eps$-neighborhood of $\partial\Omega_{R}$.
\end{corollary}

\subsubsection{Shrinking waveguides}
It turns out that the strategies and methods of this paper can be applied to partially semiclassical situations. Such limits appear for example in nanophysics when a strong anisotropic confinement is imposed or in the context of quantum waveguides with small cross section. The reader may consult \cite{ExSe89, Duclos95, Duclos05} in relation with the spectral analysis of waveguides (or \cite{KR13} in presence of magnetic fields). The partially semiclassical limits are also of crucial importance in the spectral analysis of problems with magnetic fields (see \cite{BHR15} and the book \cite{Ray14}). Nevertheless, we do not aim here at being the most general as possible on this topics and we will focus on the elementary example of bidimensional tubes shrinking in their normal direction. We can notice here that such a situation was also considered by del Pino and Felmer to investigate the Sobolev constants (see \cite{dPF96}). The result below may be considered as a more quantitative version (in two dimensions) of their result.

Let us consider a smooth and simple curve $\Gamma$ in $\R^2$ and a variable height $a : \R\to [a_{0}, a_{1}]$, with $a_{0}>0$. We assume that $a$ admits a maximum (not attained at infinity) and that $a'\in\sL^\infty(\R)$. We let
\[\forall (s,t)\in\R\times (-1,1)=\Sigma,\qquad\Phi_{h}(s,t)=\Gamma(s)+h t a(s)\n(s)\,.\]
We define the tube $\Sigma_{h}=\Phi_{h}(\Sigma)$ and we assume that $\Sigma_{h}$ does not overlap itself, i.e. that $\Phi_{h}$ is injective. Assuming in addition that the curvature is bounded, $\Phi_{h}$ is a smooth diffeomorphism as soon as $h$ is small enough. For $p\in[2,+\infty)$, we introduce 
\[
  \lambda^\Dir(\Sigma_{h}, p)
  =\lambda ((\Sigma_{h}, \mathsf{Id}, 0, 0, +\infty), 1, p)
  =\inf_{\underset{\psi\neq 0}{\psi\in \sH_{0}^1(\Sigma_{h}),}}\frac{\int_{\Sigma_{h}}|\nabla\psi|^2\dx\x}{\|\psi\|^2_{\sL^p(\Sigma_{h})}}\,.
\]

\begin{proposition}\label{prop.waveguides}
Let $p\in\left(2, 2^*\right)$. There exist $h_{0}, C>0$ such that, for all $h\in(0,h_{0})$,
\[ (1-Ch^{\frac{1}{2}})h^{-\frac{4}{p}} a^{-\frac{4}{p}}_{\max}\lambda^\Dir(\Sigma, p)\leq\lambda^\Dir(\Sigma_{h}, p)\leq (1+Ch)h^{-\frac{4}{p}} a^{-\frac{4}{p}}_{\max}\lambda^\Dir(\Sigma, p)\,,\]
where
\[\lambda^\Dir(\Sigma, p)=\lambda((\Sigma, \mathsf{Id}, 0, 0, +\infty), 1, p)\,.\]
Moreover, for all $\eps>0$ and $\rho\in\left(0,1\right)$, there exist $h_{0}>0, C>0$ such that, for all $h\in(0,h_{0})$ and all $\sL^p$-normalized minimizers $\psi$,
\[\|\psi\|_{\sL^p(\complement \mathcal{M}_{\eps})}\leq Ce^{-\eps h^{-\rho}}\,,\]
where $\mathcal{M}_{\eps}$ denotes here a $\eps$-neighborhood of the set of the maxima of $a$.
\end{proposition}

\subsection{Organization of the paper}
The paper is organized as follows. Section~\ref{sec.2} is devoted to the investigation of the Sobolev constants when the geometry is homogeneous (see Theorem~\ref{theo.0}). Under the condition that the boundary Sobolev constant is strictly less than the interior constant, we prove that the boundary constant is attained. Note that, in Section~\ref{sec.toy}, we investigate the special one-dimensional case of the half-axis with Robin condition and that we derive a condition for the existence of the minimizers. In Section~\ref{sec.3}, we prove the first estimates towards the upper bound of Theorem~\ref{theo.1}. In Section~\ref{sec.4} we introduce sliding partitions of the unity compatible with a quantum localization formula and establish the lower bound of Theorem~\ref{theo.1}. In Section~\ref{sec.5}, by combining the results of Sections~\ref{sec.3} and~\ref{sec.4}, we derive accurate $\sL^p$-localization estimates of the minimizers (Proposition~\ref{prop.Lp-loc}) and convert it into the exponential estimate of Theorem~\ref{theo.AgmonNLS}. In Section~\ref{sec.6}, we prove Proposition~\ref{prop.semicont} (and, with Propositions~\ref{prop:interior} and~\ref{prop:upperbound}, this ends the proof of the upper bound of Theorem~\ref{theo.1}). In Section~\ref{sec.6}, we also provide sufficient conditions under which Assumption~\ref{a.0} is satisfied. Finally, Section~\ref{sec.7} is devoted to the waveguide framework and we establish Proposition~\ref{prop.waveguides}. To conclude, we provide some perspectives in Section~\ref{sec.pers}.

\section{Boundary Sobolev constants with homogeneous geometry}\label{sec.2}

\subsection{A first result}
The main goal of this section is to prove the following theorem by using a variant of the concentration-compactness method (see the classical references \cite{Lions84, Lions84b,  Struwe2008, Willem1996}, or the notes by Lewin \cite{Lewin2010}). 

\begin{theorem}\label{theo.0}
Let us consider $p\in\left(2,2^*\right)$. We have the following two existence results.
\begin{enumerate}[i.]
\item\label{theo0i} If $\mathsf{G}$ is a homogeneous geometry with $U=\R^d$ and such that $\lambda(\mathsf{G},1,2)$ is positive, then the infimum $\lambda(\mathsf{G},1,p)$ is attained.
\item\label{theo0ii} If $\mathsf{G}$ is a homogeneous geometry with $U$ being an half-space and such that $\lambda(\mathsf{G},1,2)$ is positive and  
\begin{equation}\label{Int-Bord}
\lambda(\mathsf{G},1,p)<\lambda(\underline{\mathsf{G}},1,p)\,,
\end{equation}
then the infimum $\lambda(\mathsf{G},1,p)$ is attained.
\end{enumerate}
Moreover, the condition \eqref{Int-Bord} is always satisfied (for a given electro-magnetic field) as soon as $\gamma\in(-\infty,\mathsf{c}_{0})$ with $\mathsf{c}_{0}>0$ small enough.
\end{theorem}
\begin{remark}
We will only prove Theorem~\ref{theo.0}~\eqref{theo0ii}. The proof of point \eqref{theo0i} in Theorem~\ref{theo.0} (which is related to the case when $\Omega = \R^d$) is simpler and can be adapted from the proofs in \cite{EL89} (see also \cite{dCvS15}). The proof of Theorem~\ref{theo.0}~\eqref{theo0ii} will take up the following subsections. The proof of the last statement of the Theorem is given in Subsection~\ref{sec.perturb-Rob} below.
\end{remark}

\begin{remark}\label{rem.magnetictranslation}
The main difficulty in the proof of these results comes from the lack of compactness due to the action of the non-compact group of translations. Indeed, for any $\psi\in\sH^1_{\A}(U)$ and any $\x_0\in \R^d$ if $U = \R^d$ or $\x_0\in\R^{d-1}\times\{0\}$ if $U = \R^d_+$ we can define the magnetic translation 
\[
  \tau_{\x_0}\psi(\x) = e^{-i\A(\x_0)\cdot \x}\psi(\x-\x_0)
\]
which satisfies 
\[
  \mathfrak{Q}_{\mathsf{G},1}(\psi) = \mathfrak{Q}_{\mathsf{G},1}(\tau_{\x_0}\psi) \quad \text{ and } \quad \|\psi\|_{\sL^p} = \|\tau_{\x_0}\psi\|_{\sL^p}\,.
\]

Since the minimization problem $\lambda(\underline{\mathsf{G}}, h, p)$ is translation invariant, we always have: 
  \[
    \lambda(\underline{\mathsf{G}}, h, p)=\inf_{\underset{\psi\neq 0}{\psi\in \mathcal{C}^\infty_c(\R^d),}}\frac{\mathfrak{Q}_{\mathsf{G},h}(\psi)}{\|\psi\|^2_{\sL^p(\R^d)}}=\inf_{\underset{\psi\neq 0}{\psi\in \mathcal{C}^\infty_c(\R^d_{+}),}}\frac{\mathfrak{Q}_{\mathsf{G},h}(\psi)}{\|\psi\|^2_{\sL^p(\R^d_{+})}}\geq\lambda(\mathsf{G},1,p).
  \]
\end{remark}
Up to a rotation, we may assume that $U=\R^d_{+}$. Let us consider a minimizing sequence $(\psi_{j})_{j\geq 1}$ such that $\|\psi_{j}\|_{\sL^p(\Omega)}=1$. By definition, we have
\begin{equation}\label{eq.minimi}
\mathfrak{Q}_{\mathsf{G},1}(\psi_{j})\underset{j\to+\infty}{\longrightarrow}\lambda(\mathsf{G},1,p)\,.
\end{equation}
By Remark~\ref{rem.magnetictranslation}, $(\tau_{\x_j}\psi_j)_{j\geq0}$ is also a minimizing sequence, $(\x_j)_{j\geq0}$ being any sequence in $ \R^{d-1}\times \{0\}$ so that we have a loss of compactness by magnetic translations.

We overcome this difficulty thanks to the concentration-compactness principle. Our proof is divided in three steps. We show that : 
\begin{enumerate}
  \item $(\psi_j)_{j\geq0}$ is uniformly bounded in $\sH^1_{\A}(\R^d_+)$,
  \item up to magnetic translation and up to extraction $\psi_j\rightharpoonup \psi\ne0$ weakly in $\sH^1_{\A}(\R^d_+)$,
  \item  $\psi_j\to \psi$ strongly in $\sH^1_{\A}(\R^d_+)$ and  $\psi$ is a minimizer of $\lambda(\mathsf{G},1,p)$.
\end{enumerate}

\subsection{Boundedness in $\sH_{\A}^1(\R^d_{+})$}
This section is devoted to the proof of the following proposition.
\begin{proposition}\label{prop.b-H1}
Under the assumptions of Theorem~\ref{theo.0}~\eqref{theo0ii}
there exists $C>0$ such that for all $\psi\in\sH^1_{\A}(\R^d_{+})$
\[
    \|\psi\|_{\sL^2(\R^d_+)}^2 + \|\nabla|\psi|\|_{\sL^2(\R^d_+)}^2\leq \|\psi\|_{\sL^2(\R^d_+)}^2 + \|(-i\nabla+\A)\psi\|_{\sL^2(\R^d_+)}^2\leq C\mathfrak{Q}_{\mathsf{G},1}(\psi).
\]
Therefore, 
\[
  \lambda(\mathsf{G},1,p)>0
\] 
and any minimizing sequence $(\psi_{j})_{j\geq 1}$ (normalized in $\sL^p$) is bounded in $\sH^1_{\A}(\R^d_{+})$  whereas $(|\psi_{j}|)_{j\geq 1}$ is bounded in $\sH^1(\R^d_{+})$. 
 Moreover, we can assume that for all $j\geq 1$,
 \begin{equation}\label{eq.borne-L2}
\|\psi_{j}\|^2_{\sL^2(\R^d_{+})}\leq\frac{2\lambda(\mathsf{G}, 1, p)}{\lambda(\mathsf{G}, 1, 2)}\,.
\end{equation}

\end{proposition}
In order to estimate the boundary term, we will need the following two lemmas.
\begin{lemma}\label{lem.trace-eps}
There exists $C>0$ such that, for all $\eps>0$ and $\psi\in\sH^1(\R^d_{+})$, we have
\[\|\psi\|^2_{\sL^2(\R^{d-1}\times\{0\})}\leq \eps \|\nabla\psi\|^2_{\sL^2(\R^d_{+})}+C\eps^{-1}\|\psi\|^2_{\sL^2(\R^d_{+})}\,.\]
\end{lemma}
\begin{proof}
The proof is based on the elementary trace estimate:
\[\exists C>0,\quad\forall\psi\in\sH^1(\R^d_{+}),\qquad \|\psi\|^2_{\sL^2(\R^d\times\{0\})}\leq C\|\psi\|^2_{\sH^1(\R^d_{+})}\,,\]
that may be proved by density and partial integration. Then, for all $\varphi\in\sH^1(\R^{d-1}_{+})$ and $\rho>0$, we let $\psi_{\rho}(\x)=\varphi(\rho\x)$. This easily leads to
\[
  \|\varphi\|^2_{\sL^2(\R^{d-1}\times\{0\})}
  \leq C\left(\rho \|\nabla\varphi\|^2_{\sL^2(\R^d_{+})}+\rho^{-1}\|\varphi\|^2_{\sL^2(\R^d_{+})}\right)\,,
\]
and we choose $\rho=C^{-1}\eps$.
\end{proof}
\begin{lemma}\label{lem.lb}
There exists $C>0$ such that for all $\eps>0$ and all $\psi\in\sH^1(\R^d_{+})$,
\[\mathfrak{Q}_{\mathsf{G},1}(\psi)\geq (1-C\eps|\gamma|)\|(-i\nabla+\A)\psi_{j}\|^2_{\sL^2(\R^d_{+})}+(V-C|\gamma|\eps^{-1})\|\psi\|^2_{\sL^2(\R^d_{+})}\,.\]
\end{lemma}
\begin{proof}
It is a consequence of the diamagnetic inequality:
\[\forall \psi\in\sH^1_{\A}(\R^d_{+})\,,\qquad\|\nabla|\psi|\|^2_{\sL^2(\R^d_{+})}\leq \|(-i\nabla+\A)\psi\|^2_{\sL^2(\R^d_{+})}\,,\]
and of Lemma~\ref{lem.trace-eps}.
\end{proof}

We can now deduce Proposition~\ref{prop.b-H1}.
\begin{proof}
  By point \eqref{a.0i} of assumption~\ref{a.0}, we have that $\lambda(\mathsf{G},1,2)>0$ so that
  \[
    \|\psi\|_{\sL^2(\R^d_+)}^2\leq \lambda(\mathsf{G},1,2)^{-1}\mathfrak{Q}_{\mathsf{G},1}(\psi),\, \text{ for all }\psi\in \sH^1_{\A}(\R^d_+).
  \]
  Then, Lemma~\ref{lem.lb} and the diamagnetic inequality give that there is $C>0$ such that for all $\psi\in \sH^1_{\A}(\R^d_+)$
  \[
    \|\psi\|_{\sL^2(\R^d_+)}^2 + \|\nabla|\psi|\|_{\sL^2(\R^d_+)}^2\leq \|\psi\|_{\sL^2(\R^d_+)}^2 + \|(-i\nabla+\A)\psi\|_{\sL^2(\R^d_+)}^2\leq C\mathfrak{Q}_{\mathsf{G},1}(\psi).
  \]
  Finally, we deduce thanks to Sobolev's injections that 
  \[
    \lambda(\mathsf{G},1,p)>0 \text{ for any }p\in[2,2^*]
  \]
  and the conclusion follows.
\end{proof}

\subsection{Excluding the boundary vanishing}\label{sec.bnd-van}
We now focus on the following proposition (see \cite[Lemma I.1]{Lions84b}, \cite[Lemma 1.21]{Willem1996}, \cite[Lemma 2.3]{MVS2013}, \cite{VS2014}).

\begin{proposition}\label{prop.exclude-van}
Let us consider $R>0$ and consider the paving near the boundary
\[\Sigma_{R}:=\R^{d-1}\times(0,R)=\bigsqcup_{\kb\in\Z^{d-1}\times\{0\}} \Omega_{\kb, R}\,,\qquad  \Omega_{\kb, R}=[0,R]^{d-1}+R\kb\,.\]
For $q\in\left(2, 2^*\right)$ and $\psi\in\sL^2(\Sigma_{R})$, we introduce
\[M_{R}(\psi)=\sup_{\kb\in\Z^{d-1}\times\{0\}}\|\psi\|_{\sL^q(\Omega_{\kb,R})}\,.\]
For $d\geq 2$ and $R>0$, let $S>0$ be the optimal Sobolev constant for the embedding
\[\|\psi\|_{\sL^{q}(\Omega_{\zb,R})}\leq S\|\psi\|_{\sH^1(\Omega_{\zb,R})}\,.\]
Then, we have
\[\|\psi\|_{\sL^q(\Sigma_{R})}\leq S^{\frac{2}{q}}\|\psi\|^{\frac{2}{q}}_{\sH^1(\Sigma_{R})}M_{R}(\psi)^{1-\frac{2}{q}}\,.\]
\end{proposition}
\begin{proof}
We have
\[\|\psi\|^q_{\sL^q(\Sigma_{R})}=\sum_{\kb\in\Z^{d-1}\times\{0\}}\int_{\Omega_{\kb,R}} |\psi|^q \dx\x\,.\]
By Sobolev embedding, we get
\[\int_{\Omega_{\kb,R}} |\psi|^q \dx\x\leq S^2 \left(\|\psi\|^2_{\sL^2(\Omega_{\kb,R})}+\|\nabla\psi\|^2_{\sL^2(\Omega_{\kb,R})}\right)\left(\int_{\Omega_{\kb,R}}|\psi|^q\dx\x\right)^{1-\frac{2}{q}}\,.\]
We deduce that
\[
  \|\psi\|^q_{\sL^q(\Sigma_{R})}\leq S^2\|\psi\|^2_{\sH^1(\Sigma_{R})}M^{q-2}_{R}(\psi)\,.
  \qedhere
\]
\end{proof}
Let us now come back to our minimization sequence $(\psi_{j})_{j\geq 1}$ (that satisfies \eqref{eq.minimi}, by definition).
\begin{proposition}\label{prop.varphi}
We take $q=p$. There exists $R>0$, a subsequence extraction and $m_{R}>0$ such that
\[\forall j\geq 1,\qquad M_{R}(\psi_{j})\geq m_{R}>0\,.\]
Moreover, we get 
\begin{equation}\label{eq.exi-kj}
\exists(\kb_{j})_{j\geq 1}\in\Z^{d-1}\times\{0\},\qquad \forall j\geq 1,\qquad\|\psi_{j}\|_{\sL^p(\Omega_{\kb_{j},R})}\geq m_{R}\,.
\end{equation}
If we let
\[\varphi_{j}(\x):=e^{-i\A(R\kb_{j})\cdot\x}\psi_{j}(\x-R\kb_{j})\,,\]
then $(\varphi_{j})_{j\geq 1}$ is a minimizing sequence that, up to a subsequence extraction, weakly converges to some $\varphi\neq 0$ in $\sH^1_{\A}(\R^d_{+})$ equipped either with the sesquilinear form $\mathfrak{B}_{\mathsf{G},1}$ associated with $\mathfrak{Q}_{\mathsf{G},1}$ or the standard scalar product.
\end{proposition}
\begin{proof}
Let us analyze the first part of the statement. Let us assume by contradiction that, for all $R$, $\displaystyle{\lim_{j\to+\infty}M_{R}(\psi_{j})=0}$.

By applying Proposition~\ref{prop.exclude-van} to $\psi=|\psi_{j}|$ and using Proposition~\ref{prop.b-H1}, we infer that
\begin{equation}\label{eq.bd-van}
\lim_{j\to+\infty}\|\psi_{j}\|_{\sL^p(\Sigma_{R})}=0\,.
\end{equation}
This means that we are in the \enquote{boundary vanishing} situation.

Let us now introduce a partition of unity to distinguish between a neighborhood of the boundary and the interior. There exists $C>0$ such that for all $R\geq 1$ and two smooth functions on $\overline{\R_{+}}$ (depending only on the transversal variable $x_{d}$), $\chi_{1,R}$, $\chi_{2,R}$ such that 
\[\chi^2_{1,R}+\chi^2_{2,R}=1\,,\qquad |\chi'_{1,R}|^2+|\chi'_{1,R}|^2\leq CR^{-2}\,,\]
and where $\chi_{1,R}$ is a smooth and compactly supported function being $1$ for $|x_{d}|\leq \frac{R}{2}$ and $0$ for $|x_{d}|\geq R$. A well-known localization formula gives
\[\mathfrak{Q}_{\mathsf{G},1}(\psi_{j})=\sum_{k=1,2}\mathfrak{Q}_{\mathsf{G},1}(\chi_{k,R}\psi_{j})-\|\chi'_{k,R}\psi_{j}\|^2_{\sL^2(\R^d_{+})}\,.\]
It follows that, using \eqref{eq.borne-L2},
\[\mathfrak{Q}_{\mathsf{G},1}(\psi_{j})\geq \sum_{k=1,2}\mathfrak{Q}_{\mathsf{G},1}(\chi_{k,R}\psi_{j})-2CR^{-2}\frac{\lambda(\mathsf{G}, 1, p)}{\lambda(\mathsf{G}, 1, 2)}\,.\]
By a support consideration, we get
\[\mathfrak{Q}_{\mathsf{G},1}(\chi_{2,R}\psi_{j})\geq \lambda(\underline{\mathsf{G}}, 1, p)\|\chi_{2,R}\psi_{j}\|^2_{\sL^p(\R^d_{+})}\,,\]
so that there exists $C>0$ such that, for all $j\geq 1$ and $R\geq 1$,
\[\mathfrak{Q}_{\mathsf{G},1}(\psi_{j})\geq \lambda(\underline{\mathsf{G}}, 1, p)\|\chi_{2,R}\psi_{j}\|^2_{\sL^p(\R^d_{+})}-2CR^{-2}\frac{\lambda(\mathsf{G}, 1, p)}{\lambda(\mathsf{G}, 1, 2)}\,.\]
Thanks to \eqref{eq.bd-van}, we deduce that
\[\lim_{j\to+\infty}\|\chi_{1,R}\psi_{j}\|^2_{\sL^p(\R^d_{+})}=0\,,\]
so that, with the $\sL^p$-normalization of $\psi_{j}$,
\[\lambda(\mathsf{G}, 1, p)\geq  \lambda(\underline{\mathsf{G}}, 1, p)-2CR^{-2}\frac{\lambda(\mathsf{G}, 1, p)}{\lambda(\mathsf{G}, 1, 2)}\,.\]
Finally, we reach a contradiction to \eqref{Int-Bord} by choosing $R$ large enough. Therefore, the first part of the statement is now proved.

Then, \eqref{eq.exi-kj} follows by definition of $M_{R}$. The fact that $(\varphi_{j})_{j\geq 1}$ is still a minimizing sequence comes from the gauge invariance presented in Remark~\ref{rem.magnetictranslation}.

By a simple translation, we have that, for all $j\geq 1$,
\[\|\varphi_{j}\|_{\sL^p(\Omega_{\zb,R})}\geq m_{R}\,.\]
Since $(\varphi_{j})_{j\geq0}$ may be assumed (by the Banach-Alaoglu Theorem) to converge weakly (and pointwise) in $\sH^1_{\A}(\R^d_{+})$ to $\varphi$ and by compact embedding:
\[
  \|\varphi\|_{\sL^{p}(\Omega_{\zb,R})}\geq m_{R}>0\,.
  \qedhere
\]
\end{proof}

\subsection{Excluding the dichotomy}

\begin{proposition}
The function $\varphi$ of Proposition~\ref{prop.varphi} satisfies $\|\varphi\|_{\sL^p(\R^d_{+})}=1$.
\end{proposition}
\begin{proof}
By the Fatou lemma, we have $\alpha:=\|\varphi\|^p_{\sL^p(\R^d_{+})}\in(0,1]$.

We introduce $\delta_{j}=\varphi_{j}-\varphi$ for \(j \ge 1\). 
The sequence $(\delta_j)_{j \ge 1}$ weakly converges to $0$ in $\sH^1_{\A}(\R^d_{+})$ equipped with the sesquilinear form $\mathfrak{B}_{\mathsf{G},1}$ associated with $\mathfrak{Q}_{\mathsf{G},1}$. Thus  $\mathfrak{B}_{\mathsf{G},1}(\delta_{j},\varphi)\to 0$. We have
\[\mathfrak{Q}_{\mathsf{G},1}(\varphi_{j})=\mathfrak{Q}_{\mathsf{G},1}(\delta_{j})+\mathfrak{Q}_{\mathsf{G},1}(\varphi)+2\Re \mathfrak{B}_{\mathsf{G},1}(\delta_{j},\varphi)\,.\]
In other words, we can write
\begin{equation}\label{eq.splitQA}
\mathfrak{Q}_{\mathsf{G},1}(\varphi_{j})=\mathfrak{Q}_{\mathsf{G},1}(\delta_{j})+\mathfrak{Q}_{\mathsf{G},1}(\varphi)+\eps_{j}\,,
\end{equation}
with $\eps_{j}\to 0$.

We must prove that the $\sL^p$ norm also splits into two parts:
\begin{equation}\label{eq.splitLp}
\|\varphi_{j}-\varphi\|^p_{\sL^p(\R^d_{+})}+\|\varphi\|^p_{\sL^p(\R^d_{+})}-\|\varphi_{j}\|^p_{\sL^p(\R^d_{+})}=\Tilde\eps_{j}\to 0\,.
\end{equation}
Let us temporarily assume that \eqref{eq.splitLp} holds. Thanks to \eqref{eq.splitQA}, and using \eqref{eq.splitLp},
\begin{align*}
\mathfrak{Q}_{\mathsf{G},1}(\varphi_{j})&\geq \lambda(\mathsf{G}, 1, p)\left(\|\delta_{j}\|^2_{\sL^p(\R^d_{+})}+\|\varphi\|^2_{\sL^p(\R^d_{+})}\right)+\eps_{j}\,,\\
&= \lambda(\mathsf{G}, 1, p)\left(\left(1-\alpha+\Tilde\eps_{j}\right)^{\frac{2}{p}}+\alpha^{\frac{2}{p}}\right)+\eps_{j}\,.
\end{align*}
Since $(\varphi_{j})_{j\geq 1}$ is a minimizing sequence, we get
\[\lambda(\mathsf{G}, 1, p)\geq \lambda(\mathsf{G}, 1, p)\left((1-\alpha)^{\frac{2}{p}}+\alpha^{\frac{2}{p}}\right)\,.\]
But we have $\lambda(\mathsf{G}, 1, p)>0$ so that
\[(1-\alpha)^{\frac{2}{p}}+\alpha^{\frac{2}{p}}\leq 1\,,\text{ with }\alpha\in(0,1]\,.\]
Since $p>2$ and by strict convexity, we must have $\alpha=1$. Therefore we conclude that $\|\varphi\|_{\sL^p(\R^d_{+})}=1$. This finishes the proof of the proposition, modulo the proof of \eqref{eq.splitLp}. For that purpose we write
\[\Tilde\eps_{j}:=\int_{\R^d_{+}} |\varphi_{j}-\varphi|^p-|\varphi_{j}|^p+|\varphi|^p\dx\x\,.\]
Let us prove that the sequence $(|\varphi_{j}-\varphi|^p-|\varphi_{j}|^p)_{j\geq 1}$ is equi-integrable on $\R^d_{+}$. There exists $C(p)>0$ such that,
\[\left||\varphi_{j}-\varphi|^p-|\varphi_{j}|^p\right|\leq C(p)(|\varphi_{j}|^{p-1}+|\varphi|^{p-1})|\varphi|\,.\]
For $R>0$, by the H\"older inequality, we get
\[\int_{|\x|\geq R} |\varphi_{j}|^{p-1}|\varphi|\dx\x\leq\left(\int_{|x|\geq R}|\varphi_{j}|^p\dx\x\right)^{\frac{p-1}{p}}\left(\int_{|x|\geq R}|\varphi|^p\dx\x\right)^{\frac{1}{p}}\leq\left(\int_{|x|\geq R}|\varphi|^p\dx\x\right)^{\frac{1}{p}} \,.\]
Thus, for all $\eps>0$, there exists $R>0$, such that for all $j\geq 1$, we have
\[\left|\int_{|\x|\geq R} |\varphi_{j}-\varphi|^p-|\varphi_{j}|^p+|\varphi|^p\dx\x\right|\leq\frac{\eps}{2}\,.\]
This proves the equi-integrability. Now the embedding $\sH^1(\mathcal{B}(0,R))\subset \sL^p(\mathcal{B}(0,R))$ is compact so that the sequence $(\varphi_{j})_{j\geq 1}$ strongly converges to $\varphi$ in $\sL^p(\mathcal{B}(0,R))$ and thus, for $j\geq j(R,\eps)$,
\[\left|\int_{|\x|\leq R} |\varphi_{j}-\varphi|^p-|\varphi_{j}|^p+|\varphi|^p\dx\x\right|\leq\frac{\eps}{2},.\]
This implies that $|\Tilde\eps_{j}|\leq \eps$.

\end{proof}

\begin{proof}[Proof of Theorem~\ref{theo.0}~\eqref{theo0ii}]
To finish the proof of Theorem~\ref{theo.0}~\eqref{theo0ii}, it remains to notice that
\[\lambda(\mathsf{G}, 1, p)=\liminf_{j\to+\infty}\mathfrak{Q}_{\mathsf{G},1}(\varphi_{j})\geq \mathfrak{Q}_{\mathsf{G},1}(\varphi)\geq \lambda(\mathsf{G}, 1, p)\|\varphi\|^2_{\sL^p(\R^d_{+})}=\lambda(\mathsf{G}, 1, p)\,,\]
and thus $\varphi$ is a minimizer. 
\end{proof}

\subsection{Exponential estimates}
When the minimizers exist, they satisfy decay estimates of Agmon type.
\begin{proposition}\label{prop.exp}
If $\mathsf{G}$ is a homogeneous geometry with $\lambda(\mathsf{G}, 1, 2)>0$ and if the infimum \eqref{eq.Sobo} is attained, then, for all minimizer $\psi$, there exists $\alpha>0$ such that
\[e^{\alpha|\x|}\psi\in\sL^2(U)\,,\qquad \mathfrak{Q}_{\mathsf{G},h}(e^{\alpha|\x|}\psi)<+\infty\,.\]
\end{proposition}
\begin{proof}
We only consider the case $U=\R^d_{+}$. Let us consider a minimizer $\psi_{0}$ and the nonlinear potential $V_{\mathsf{NL}}=-\lambda(\mathsf{G}, 1, p)|\psi_{0}|^{p-2}$. We have $V_{\mathsf{NL}}\in\sL^{\frac{p}{p-2}}(U)$. The corresponding quadratic form is defined on the space $\sH^1_{\mathsf{A}}(U)$ by
\begin{equation}\label{eq.QNL}
  \mathfrak{Q}_{\mathsf{G}, 1, \mathsf{NL}}(\psi)=\mathfrak{Q}_{\mathsf{G}, 1}(\psi)+\int_{U}V_{\mathsf{NL}}|\psi|^2\dx\x\,.
\end{equation}
By the H\"older inequality, we see that
\[\int_{U}|V_{\mathsf{NL}}||\psi|^2\dx\x\leq \|V_{\mathsf{NL}}\|_{\sL^{\frac{p}{p-2}}(U)}\|\psi\|^2_{\sL^p(U)}\,\]
and thus, by Sobolev embedding (and homogeneity) and the diamagnetic inequality, for all $\eps>0$, there exists $C_{\eps}>0$ such that, for all $\psi\in\sH^1_{\mathsf{A}}(U)$,
\begin{align}\label{eq.controlNL}
\nonumber\int_{U}|V_{\mathsf{NL}}||\psi|^2\dx\x&\leq C \|V_{\mathsf{NL}}\|_{\sL^{\frac{p}{p-2}}(U)}(\eps\|\nabla|\psi|\|^2+C_{\eps}\|\psi\|^2_{\sL^2(U)})\,.\\
&\leq C \|V_{\mathsf{NL}}\|_{\sL^{\frac{p}{p-2}}(U)}(\eps\|(-i\nabla+\mathsf{A})\psi\|^2_{\sL^2(U)}+C_{\eps}\|\psi\|^2_{\sL^2(U)})\,.
\end{align}
We infer that there exists $\Tilde C>0$ such that for all $\eps>0$ there exists $\Tilde C_{\eps}>0$ such that for all $\psi\in\sH^1_{\mathsf{A}}(U)$,
\[\mathfrak{Q}_{\mathsf{G}, 1, \mathsf{NL}}(\psi)\geq (1-\Tilde C\eps)\|(-i\nabla+\mathsf{A})\psi\|^2_{\sL^2(U)}+\mathsf{V}\|\psi\|^2_{\sL^2(U)}+\mathsf{c}\|\psi\|^2_{\partial U}-\Tilde C_{\eps}\|\psi\|^2_{\sL^2(U)}\,,\]
and, by using Lemma~\ref{lem.trace-eps} and again the diamagnetic inequality, it follows that
\[\mathfrak{Q}_{\mathsf{G}, 1, \mathsf{NL}}(\psi)\geq (1-\hat C\eps)\left\{\|(-i\nabla+\mathsf{A})\psi\|^2_{\sL^2(U)}+\mathsf{V}\|\psi\|^2_{\sL^2(U)}\right\}-\hat C_{\eps}\|\psi\|^2_{\sL^2(U)}\,.\]
This proves that $\mathfrak{Q}_{\mathsf{G}, 1, \mathsf{NL}}$ is bounded from below on $\sH^1_{\mathsf{A}}(U)$ and thus defines a self-adjoint operator $\mathfrak{L}_{\mathsf{G}, 1,\mathsf{NL}}$. The function $\psi_{0}$ belongs to the domain of this operator and satisfies $\mathfrak{L}_{\mathsf{G}, 1,\mathsf{NL}}\psi_{0}=0$. Now, the exponential decay will be established if we prove that
\begin{multline}\label{eq.min-infty}
\exists C>0\,, \forall \eps>0\,,\exists R>0\,,\forall \psi\in\sH^1_{\mathsf{A}}(U)\,,\\
\supp(\psi)\subset\complement\mathcal{B}(0,R)\Longrightarrow\mathfrak{Q}_{\mathsf{G}, 1,\mathsf{NL}}(\psi)\geq (1-C\eps)\mathfrak{Q}_{\mathsf{G}, 1}(\psi)-C\eps\|\psi\|^2_{\sL^2(U)}\,.
\end{multline}
Indeed, this implies that, for all $\psi\in\sH^1_{\mathsf{A}}(U)$ with $\supp(\psi)\subset\complement\mathcal{B}(0,R)$,
\[\mathfrak{Q}_{\mathsf{G}, 1,\mathsf{NL}}(\psi)\geq (1-C'\eps)\lambda(\mathsf{G}, 1, 2)\|\psi\|^2_{\sL^2(U)}\,.\]
From this we can deduce, by using Persson's theorem (see \cite{Persson60}), that we have $\inf\mathsf{sp}_{\mathsf{ess}}(\mathfrak{L}_{\mathsf{G}, 1,\mathsf{NL}})\geq\lambda(\mathsf{G}, 1, 2)>0$ and the conclusion follows by using the Agmon-Persson estimates (see \cite{Agmon85} or for instance \cite[Proof of Proposition 10.10]{Ray14}). From the proof of these estimates, we may even find an $\alpha>0$ common to all the minimizers $\psi$.

Therefore, let us explain where \eqref{eq.min-infty} comes from. For that purpose, we come back to \eqref{eq.controlNL} with $\eps=1$ and we notice that, for all $R\geq 0$ and $\psi\in\sH^1_{\mathsf{A}}(U)$ such that $\supp(\psi)\subset\complement\mathcal{B}(0,R)$,
\[\int_{U}|V_{\mathsf{NL}}||\psi|^2\dx\x\leq C \|V_{\mathsf{NL}}\|_{\sL^{\frac{p}{p-2}}(\complement\mathcal{B}(0,R))}(\|(-i\nabla+\mathsf{A})\psi\|^2_{\sL^2(U)}+C_{1}\|\psi\|^2_{\sL^2(U)})\,.\]
Since $V_{\mathsf{NL}}\in \sL^{\frac{p}{p-2}}(U)$, $\|V_{\mathsf{NL}}\|_{\sL^{\frac{p}{p-2}}(\complement\mathcal{B}(0,R)))}$ goes to zero when $R$ goes to infinity. Then, from \eqref{eq.QNL} (and again Lemma \ref{lem.trace-eps} with the diamagnetic inequality to control the boundary term), we deduce \eqref{eq.min-infty}.
\end{proof}

\subsection{A sufficient condition for boundary attraction}\label{sec.perturb-Rob}
This section is devoted to the proof of the last part of Theorem~\ref{theo.0}.
\begin{proposition}\label{prop:ccprinciple}
If $\mathsf{G}$ is a homogeneous geometry with $U$ being a half-space and with fixed $(\mathsf{V},\A)$, then there exists $\mathsf{c}_{0}>0$, such that for $\mathsf{c}\in(-\infty,\mathsf{c}_{0})$, we have
\[\lambda(\mathsf{G},1,p)<\lambda(\underline{\mathsf{G}},1,p)\,.\]
\end{proposition}
\begin{proof}
  Let us first prove the inequality in the case  $\mathsf{G} = (\R^d_+,\Id,V,\A ,0)$. Let $u_0\in\sH^1_{\A}(\R^d)$ be a minimizer of $\lambda(\underline{\mathsf{G}},1,p)$ given by point \eqref{theo0i} Theorem~\ref{theo.0} such that $\|u_0\|_{\sL^p(\R^d)} = 1$. Up to a translation in the $e_d = (0,\dots,0,1)$ direction and up to the symmetry $\x \mapsto-\x$, we can assume that 
  \[
    \|u_0\|_{\sL^p(\R^d_+)}^p = \|u_0\|_{\sL^p(\R^d_-)}^p = \frac{1}{2}
  \]
  and
  \[
  \int_{\R^d_+}|(-i\nabla+\mathsf{A})u_0|^2+\mathsf{V}|u_0|^2\dx\x\leq \int_{\R^d_-}|(-i\nabla+\mathsf{A})u_0|^2+\mathsf{V}|u_0|^2\dx\x\,.
  \]
  Then, we get
  \begin{eqnarray*}
    &\lambda(\underline{\mathsf{G}},1,p) &= \int_{\R^d}|(-i\nabla+\mathsf{A})u_0|^2+\mathsf{V}|u_0|^2\dx\x\\
    &&\geq 2\int_{\R^d_+}|(-i\nabla+\mathsf{A})u_0|^2+\mathsf{V}|u_0|^2\dx\x\\
    &&\geq 2\|u_0\|_{\sL^p(\R^d_+)}^{2}\lambda(\mathsf{G},1,p) = 2^{1-2/p}\lambda(\mathsf{G},1,p)\,.
  \end{eqnarray*}
  Thus, we are left with the case $c\ne0$. Let us remark that 
  \[
    c \mapsto \lambda(\mathsf{G},1,p)
  \]
   is a non-negative, concave and non-decreasing function of $c$ since it is a infimum of non-negative, affine and non-decreasing functions. Hence, we get the result provided that $\lambda(\mathsf{G},1,p)\leq \lambda(\underline{\mathsf{G}},1,p)$ for any $c>0$. To do so, we build a sequence made of magnetic translated of $u_0$ in the $e_d$ direction that are multiplied by a cut-off function so as to vanish on $\R^{d-1}\times \{0\}$. 
\end{proof}

\subsection{Study of a one-dimensional model}\label{sec.toy}
In the last section, we have seen that the existence of the minimizers in presence of a Robin boundary is ensured if the Robin parameter is not too large. Actually, in dimension one (without electric or magnetic field), we can prove that, above a certain intensity, the minimizers do not exist (as we will see in the following lines).

We are interested in the map $c\mapsto\lambda((\R_+,\Id,1,0,c),1,p)$. The goal is to get a good understanding of the phenomena occurring, studying the simplest case when the concentration-compactness principle is not needed. Indeed, we look for a real-valued solution of the following ordinary differential equation problem:
\begin{equation}
\left\{
\begin{aligned}\label{eq:Robin1}
  -u'' + u & = \lambda |u|^{p-2}u && \text{ in } \R_+,\\
  u'(0) & = cu(0),\\
 \|u\|_{\sL^p(\R_+)} & = 1,
\end{aligned}\right.
\end{equation}
where $\lambda =  \lambda((\R_+,\Id,1,0,c),1,p)$, $u\in \sH^1(\R_+,\R)$, $p>2$ and $c\in\R$.
We get the following result.
\begin{proposition}\label{prop:Robin}
  The system \eqref{eq:Robin1} has a unique solution for $c\in(-1,1)$ and no solution for $|c|\geq 1$.
  Moreover, we have
  \begin{enumerate}[\rm (i)]
    \item $\lambda((\R_+,\Id,1,0,c),1,2)>0$ if and only if $c>-1$,
    \item $\lambda((\R_+,\Id,1,0,c),1,p)<\lambda((\R,\Id,1,0,0),1,p)$ for all $c\in(-1,1)$,
    \item $\lambda((\R_+,\Id,1,0,c),1,p)=\lambda((\R,\Id,1,0,0),1,p)$ for all $c\geq1$.
  \end{enumerate}
\end{proposition}
We split our study into two steps: 
\begin{enumerate}[i.]
  \item Study of the Cauchy problem \eqref{eq:Robin1} with $\lambda=1$ and $c\in \R$ fixed 
  but without the restriction $\|u\|_{\sL^p(\R_+)} = 1$.
  \item Study of the dependence of the solutions of \eqref{eq:Robin1} 
  to describe the behavior of the function $c\mapsto\lambda((\R_+,\Id,1,0,c),1,p)$.
\end{enumerate}
\subsubsection{First step}
Let us remark that up to the change of unknown $u\leadsto u\lambda^{\frac{1}{p-2}}$, 
the system \eqref{eq:Robin1} without the constraint on the integral is equivalent to
\begin{equation}
\left\{
\begin{split}\label{eq:Robin2}
  &u' = v\\
  &v'  = u -  |u|^{p-2}u  \\
  &v(0) = cu(0)\,.
\end{split}
\right.
\end{equation}
Obviously, we are only interested in nontrivial solutions of \eqref{eq:Robin2}  
so that without loss of generality, we can assume that $u(0)>0$.
This is a Hamiltonian system 
\begin{equation}
\left\{
\begin{split}
  &u' = \frac{\partial H}{\partial v} (u,v) \\
  &v' = - \frac{\partial H}{\partial u} (u,v)
\end{split}
\right.
\end{equation}
where the Hamiltonian \(H\) is defined by 
\[
  H(u,v)  := \frac{|v|^2 -|u|^2}{2} + \frac{|u|^p}{p}\,.
\]
As a consequence, we immediately get that $\partial_r H(u(r),v(r)) = 0$.
Let us notice that $H$ is coercive: 
\[
\lim_{\|(u,v)\|\to +\infty} H(u,v) = +\infty\,,
\]
so that all solutions of \eqref{eq:Robin2} are global. 
Moreover, since we are looking for a solution $u\in \sH^1(\R_+,\R)$ \emph{i.e.}\ such that 
\[
  \int_{\R_+}(|u|^2+|v|^2)\dx r<+\infty\,,
\]
the initial condition has to satisfy $H(u(0),v(0)) = 0$.  This follows from the continuity of $H$ since for any $E\ne 0$, there exists $R>0$ such that
\[
  H^{-1}(\{E\})\cap B(0,R) = \emptyset
\]
where $B(0,R)$ is the open euclidian ball of radius $R$ centered in $(0,0)$.

\begin{figure}\label{fig.1}
\includegraphics[width=0.6\textwidth]{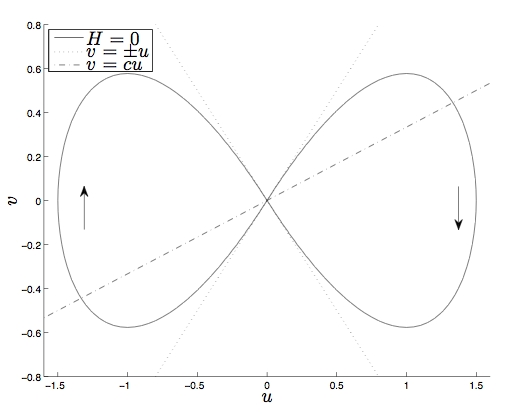}
\caption{Zero set of $H$}
\end{figure}

Thus, we obtain the following lemma.
\begin{lemma}
  \label{lem:Robin1}
  We have: 
  \begin{enumerate}[\rm (i)]
    \item for $|c|\geq1$, there is no nontrivial solution of \eqref{eq:Robin2},
    \item for $|c|<1$, there is a unique $u_{c}^0>0$ such that $H(u_{c}^0,cu_{c}^0)=0$, the associated solution $(u_c,v_c)$ tends to $(0,0)$ at infinity and satisfies $u_c(r)>0$ for all $r\geq0$. 
  \end{enumerate}
\end{lemma}
\begin{proof}
  The equation
  \[
    0 = H(u^0,cu^0) = \frac{c^2-1}{2}|u^0|^2 + \frac{|u^0|^p}{p}
  \]
  has no nontrivial solution for $|c|\geq1$ and has a unique solution $u_c^0>0$ for $|c|<1$. Moreover, $(0,0)$ is the unique critical point of $H^{-1}(\{0\})$ and the conclusion follows from the Cauchy-Lipschitz theorem.  
\end{proof}
Let us study now the decay at infinity of $u_c$ and $v_c$ for $|c|<1$.
\begin{lemma}
  \label{lem:Robin2}
  Let $c\in(-1,1)$. Then, $(u_c,v_c)$ decays exponentially at infinity and $u_c$ belongs to $\sH^1(\R_+,\R)$. 
\end{lemma}
\begin{proof}
 
  By Lemma~\ref{lem:Robin1}, $\arctan(v_c/u_c)$ is well-defined on $\R_+$ and satisfies
  \[
    \begin{split}
    \frac{d}{dr} \arctan (v_c/u_c) &= \frac{{v_c}'u_c-{u_c}'v_c}{u_c^2+v_c^2} = \frac{-(1-2/p)u_c^p}{u_c^2+v_c^2}< 0
    \end{split}
  \]
  so that $ \arctan (v_c/u_c)$ decays from $\arctan(c)$ to $-\pi/4$.  Hence, there exists $T\geq0$ such that $-u_c(r)\leq v_c(r)\leq \frac{-u_c(r)}{2}$ for all $r\geq T$ and
  \[
    \frac{d}{dr}|u_c|^p = p|u_c|^{p-2}u_c{u_c}' = p|u_c|^{p-2}u_cv_c\leq -\frac{p}{2}|u_c|^{p}.
  \]
  Thus, we obtain that $|u_c(r)|^p\leq C\exp(-\frac{p}{2}r)$ for all $r\geq 0$ and the conclusion follows
\end{proof}
\subsubsection{Second step}
In the following lemma, we study the dependence of $\sL^p$-norm of $u_c$ from $c$.
\begin{lemma}
  \label{lem:Robin3}
  We have that
  \[
    c\in(-1,1)\mapsto \|u_c\|^p_{\sL^p(\R_+)}\in \R_+
  \]
  is strictly increasing.
\end{lemma}
\begin{proof}
  Let $-1<c'<c<1$. Let us write
  \[
    T_{c,c'} := \inf\{r>0, \arctan(v_c(r)/u_c(r))<\arctan(c')\}\,.
  \]
  The quantity $T_{c,c'}>0$ is well-defined from the proof of Lemma~\ref{lem:Robin2}. 
  By the uniqueness of the solution in the Cauchy-Lipschitz theorem, we get that 
  \[
    (u_c,v_c)( \cdot + T_{c,c'})=(u_{c'},v_{c'})(\cdot)
  \]
  since $H(u_c,v_c)(T_{c,c'})=0$ and $v_c(T_{c,c'}) = c'u_c(T_{c,c'})$. This ensures that
  \[
    \begin{split}
      \|u_c\|^p_{\sL^p(\R_+)} &= \int_0^{T_{c,c'}}|u_c|^p\dx r + \int_{T_{c,c'}}^{+\infty}|u_c|^p \dx r> \int_0^{+\infty}|u_{c'}|^p \dx r\,.
    \end{split}
    \qedhere
  \]
\end{proof}
Let us introduce for $c\in(-1,1)$
\[
  \lambda_c := \frac{\|u_c\|_{\sH^1(\R_+)^2}+c|u_c(0)|^2}{\|u_c\|^2_{\sL^p(\R_+)}} = \|u_c\|^{p-2}_{\sL^p(\R_+)}\geq \lambda((\R_+,\Id,1,0,c),1,p).
\]
Lemma~\ref{lem:Robin3} ensures that $c\mapsto \lambda_c$ is strictly increasing.
It remains to study the limiting behavior of $\lambda_c$ at $\pm 1$.
\begin{lemma}
  \label{lem:Robin4}
  We have 
  \[
    \lim_{c\to1} \lambda_c = \lambda((\R,\Id,1,0,0),1,p)\quad\text{ and }\quad \lim_{c\to -1}\lambda_c = 0\,,
  \]
  so that, for all $c\in(-1,1)$,
  \[
    \lambda_c = \lambda((\R_+,\Id,1,0,c),1,p)\,.
  \]
\end{lemma}
\begin{proof}
  Let us denote by
  \[
    T_c := T_{c,0} = \inf\{r>0 : \arctan(v_c(r)/u_c(r))<0\}\,.
  \]
  By the Cauchy-Lipschitz theorem, we get that
  \[
    (u_c,v_c)(r + T_c) = (u,u')(r)\text{ for all } r\in[-T_c,+\infty)
  \]
  where $(u,v)$ is the solution of
  \[
    -u'' +u = |u|^{p-2}u\text{ on } \R 
  \]
  such that $u(0) = u_0^0$ and $ u'(0) = 0$ with $u_0^0$ given by Lemma~\ref{lem:Robin1}. Let us remark that Theorem~\ref{theo.0} ensures that
  \[
    \lambda((\R,\Id,1,0,0),1,p) = \|u\|^{p-2}_{\sL^p(\R)}\,.
  \]
  We also have that $\lim_{c\to+1}T_c = +\infty$ and 
  \[
    \lim_{c\to1}\lambda_c^{\frac{p}{p-2}} =  \lim_{c\to1} \int_{-T_c}^\infty|u|^p\dx r = \int_{-\infty}^\infty|u|^p\dx r = \lambda((\R,\Id,1,0,0),1,p)^{\frac{p}{p-2}}
  \]
  since $(0,0)$ is an equilibrium of \eqref{eq:Robin2}. The same ideas give that
  \[
    \lim_{c\to-1} \lambda_c = 0\,.
  \]
  Finally, Theorem~\ref{theo.0} and Lemma~\ref{lem:Robin3} allow us to conclude.
\end{proof} 
Let us end this section with the proof of Proposition~\ref{prop:Robin}.

\begin{proof}[Proof of Proposition~\ref{prop:Robin}]
The first point follows from a standard spectral analysis (for negative $c$ there exists only one eigenvalue below the essential spectrum that is $1-c^2$). The second point follows from Lemmas~\ref{lem:Robin3} and~\ref{lem:Robin4}. The third point is an immediate consequence of the first point, Lemma~\ref{lem:Robin1} and Theorem~\ref{theo.0}.
\end{proof}
\section{Upper bounds of $\lambda(\mathcal{G}, h, p)$}\label{sec.3}
The upper bound in Theorem~\ref{theo.1} will be proved by inserting appropriate test functions in the Sobolev quotient: either functions localized inside the domain, or functions localized near the boundary. Of course, the case related to the boundary is slightly more delicate and involves a local straightening of the boundary. Anyway, after an appropriate rescaling, we will locally see the appearence of the concentration function $\x\mapsto \lambda(\mathcal{G}_{\x}, 1, p)$. Thus, we will have to select a minimal point of this function. The existence of such a point $\x_{0}\in\overline{\Omega}$ is ensured by Proposition~\ref{prop.semicont} which we will prove in Section~\ref{sec.5}.

Depending on whether $\x_{0}\in\Omega$ or $\x_{0}\in\partial\Omega$, this section is divided into two parts and devoted to the proof of Propositions~\ref{prop:interior} and~\ref{prop:upperbound}.

\subsection{Interior estimate}
Here is the estimate related to the interior contribution.
\begin{proposition}\label{prop:interior}
Let $\x_{0}\in\Omega$. There exists $h_{0}>0, C>0$ such that, for all $h\in(0,h_{0})$,
\[\lambda(\mathcal{G}, h, p)\leq h^{\frac{d}{2}-\frac{d}{p}}h\left(\lambda(\mathcal{G}_{\x_{0}},1,p)+Ch^{\frac{1}{2}}\right)\,. \]
\end{proposition}

\begin{proof}
Let us consider a smooth cutoff function $0\leq \chi_{0}\leq 1$ being $1$ in $\mathcal{B}(\x_{0},\eps_{0})$ and being zero away from $\mathcal{B}\left(\x_{0},2\eps_{0}\right)$. It follows from \eqref{eq.Lorentz}, \eqref{eq.jaugex0} and \eqref{eq.linear-approx} that there exists a smooth real function $\varphi_{0}$ such that, on the support of $\chi_{0}$,
\begin{equation}\label{eq.Taylor}
|\A(\x)-\nabla\varphi_{0}(\x)-\mathcal{A}^{\mathsf{L}}_{\x_{0}}(\x)|\leq C|\x-\x_{0}|^2\,.
\end{equation}
Let us consider an $\sL^p$-normalized minimizer $\Psi_{\x_{0}}$ associated with $\lambda(\mathcal{G}_{\x_{0}},1,p)$ and let $\Tilde{\Psi}_{\x_{0}}(\cdot)=\Psi_{\x_{0}}(\x_{0}+\cdot)$. We let
\[\psi_{h}(\x)=h^{-\frac{d}{2p}}\chi_{0}(\x)e^{-i\frac{\varphi_{0}(\x)}{h}}\Tilde{\Psi}_{\x_{0}}(h^{-\frac{1}{2}}(\x-\x_{0}))\,.\]
We notice that
\begin{align*}
\|\psi_{h}\|^2_{\sL^p(\Omega)}&=h^{-\frac{d}{p}}\left(\int_{\R^d} \left|\chi_{0}(\x)\right|^p\left|\Tilde{\Psi}_{\x_{0}}(h^{-\frac{1}{2}}(\x-\x_{0}))\right|^p\dx\x\right)^{\frac{2}{p}}\\
&=\left(1-\int_{\complement\mathcal{B}(\zb,\eps_{0} h^{-\frac{1}{2}})}\left(1-  \left|\chi_{0}(\x_{0}+h^{\frac{1}{2}}\y)\right|^p\right)\left|\Tilde{\Psi}_{\x_{0}}(\y)\right|^p\dx\y\right)^{\frac{2}{p}}
\end{align*}
and thus, thanks to Proposition~\ref{prop.exp},
\begin{equation}\label{eq.L2-psih}
\|\psi_{h}\|^2_{\sL^p(\Omega)}=1+\mathcal{O}(h^\infty)\,.
\end{equation}
Then, we estimate $\mathfrak{Q}_{\mathcal{G},h}(\psi_{h})$. Thanks to the localization formula, we have
\begin{multline}\label{eq.ubQ}
\int_{\Omega} |(-ih\nabla+\A)\psi_{h}|^2\dx\x\\
=h^{-\frac{d}{p}}\!\!\!\int_{\Omega} \left|\chi_{0}(\x)\right|^2\left|(-ih\nabla+\A)e^{-i\frac{\varphi_{0}(\x)}{h}}\Tilde{\Psi}_{\x_{0}}\left(h^{-\frac{1}{2}}(\x-\x_{0})\right)\right|^2\!\!\!\dx\x\\
+h^{2-\frac{d}{p}}\int_{\R^d}|\nabla\chi_{0}(\x)|^2\left|\Tilde{\Psi}_{\x_{0}}\left(h^{-\frac{1}{2}}(\x-\x_{0})\right)\right|^2\dx\x\,.
\end{multline}
By support considerations, and by using again Proposition~\ref{prop.exp}, we get
\begin{equation}\label{eq.reste-com}
\int_{\R^d}|\nabla\chi_{0}(\x)|^2\left|\Tilde{\Psi}_{\x_{0}}\left(h^{-\frac{1}{2}}(\x-\x_{0})\right)\right|^2\dx\x=\mathcal{O}(h^{\infty})\,.
\end{equation}
We have
\begin{multline}\label{eq.ubQ'}
\int_{\Omega} \left|\chi_{0}(\x)\right|^2\left|(-ih\nabla+\A)e^{-i\frac{\varphi_{0}(\x)}{h}}\Tilde{\Psi}_{\x_{0}}\left(h^{-\frac{1}{2}}(\x-\x_{0})\right)\right|^2\dx\x\\
\leq \int_{\mathcal{B}(\x_{0},2\eps_{0})} \left|(-ih\nabla+\A-\nabla\varphi_{0})\Tilde{\Psi}_{\x_{0}}\left(h^{-\frac{1}{2}}(\x-\x_{0})\right)\right|^2\dx\x\,.
\end{multline}
With \eqref{eq.Taylor}, we get, for all $\eta>0$,
\begin{multline}\label{eq.ubQ''}
\int_{\mathcal{B}(\x_{0},2\eps_{0})} \left|(-ih\nabla+\A-\nabla\varphi_{0})\Tilde{\Psi}_{\x_{0}}\left(h^{-\frac{1}{2}}(\x-\x_{0})\right)\right|^2\dx\x\\
\leq (1+\eta)\int_{\R^d} \left|(-ih\nabla+\mathcal{A}^{\mathsf{L}}_{\x_{0}}(\x))\Tilde{\Psi}_{\x_{0}}\left(h^{-\frac{1}{2}}(\x-\x_{0})\right)\right|^2\dx\x\\
+C^2(1+\eta^{-1})\int_{\R^d}|\x-\x_{0}|^4 \left|\Tilde{\Psi}_{\x_{0}}\left(h^{-\frac{1}{2}}(\x-\x_{0})\right)\right|^2\dx\x\,.
\end{multline}
With the definitions of $\Tilde{\Psi}_{\x_{0}}$ and $\Psi_{\x_{0}}$ and Proposition~\ref{prop.exp}, we deduce
\begin{equation}\label{eq.x4}
\int_{\R^d}|\x-\x_{0}|^4 \left|\Tilde{\Psi}_{\x_{0}}\left(h^{-\frac{1}{2}}(\x-\x_{0})\right)\right|^2\dx\x\leq D h^{2}h^{\frac{d}{2}}\,.
\end{equation}
Similarly, we get
\begin{equation}\label{eq.apV}
\int_{\Omega} |V(\x)-V(\x_{0})||\psi_{h}(\x)|^2\dx\x\leq C h^{-\frac{d}{p}}h^{\frac{1}{2}}h^{\frac{d}{2}}\,.
\end{equation}
We choose $\eta=h^{\frac{1}{2}}$. By combining \eqref{eq.ubQ}, \eqref{eq.reste-com}, \eqref{eq.ubQ'}, \eqref{eq.ubQ''}, \eqref{eq.x4} and \eqref{eq.apV}, we infer from the definition of  $\Tilde{\Psi}_{\x_{0}}$ and $\Psi_{\x_{0}}$ that
\begin{equation}\label{eq.Q-psih}
\mathfrak{Q}_{\mathcal{G},h,p}(\psi_{h})\leq  h^{\frac{d}{2}-\frac{d}{p}}h\left(\lambda(\mathcal{G}_{\x_{0}},1,p)+Ch^{\frac{1}{2}}\right)\,.
\end{equation}
The conclusion follows by combining \eqref{eq.L2-psih} and \eqref{eq.Q-psih}.
\end{proof}

\subsection{Boundary estimate}

\subsubsection{The electro-magnetic Robin Laplacian near the boundary}\label{sec.coord}
Let us describe the geometry near the boundary of $\Omega$. Since $\partial\Omega$ is smooth, we may consider a covering $\left(\mathcal{B}(\mathsf{X}_{\ell},r)\right)_{1\leq\ell\leq N}$ of $\partial\Omega$ such that the following holds. For each $\ell\in\{1,\cdots, N\}$, there exists a smooth parametrization $\Phi_{\ell} : U_{\ell}\times(0,t_{\ell})\to\Omega\cap\mathcal{B}(\mathsf{X}_{\ell},r)$, where $U_{\ell}\subset\R^{d-1}$ is an open set with $0\in U_{\ell}$ and $\Phi_{\ell}(0,0)=\mathsf{X}_{\ell}$.

For $\x_{0}\in\partial\Omega$, we have
\[\mathsf{T}_{\x_{0}}(\partial\Omega)=\left(d\Phi_{\ell}\right)_{\y_{0}}(\R^{d-1}\times\{0\})\,,\qquad \x_{0}=\Phi_{\ell}(\y_{0})\,,\quad \y_{0}=(s_{0},0)\,.\]
We also recall that the metric induced by $\Phi_{\ell}$ is in the form
\[G_{\ell}=(d \Phi_{\ell})^{\mathsf{T}}\dx \Phi_{\ell}\,,\]
and we let 
\[|G_{\ell}|=\det G_{\ell},\quad g_{\ell}(s)=G_{\ell}(s,0)\,,\quad |g_{\ell}|=\det g_{\ell}\,.\]
Then, we shall discuss the expression of our Laplacian in these coordinates. If $\psi$ is supported in $\Omega\cap\mathcal{B}(\mathsf{X}_{\ell},r)$, we may use the change of variables $\x=\Phi_{\ell}(s,t)$ and we get
\[\|\psi\|^2_{\sL^p(\Omega)}=\left(\int_{U_{\ell}\times(0,t_{\ell})}|\Tilde\psi|^p|G_{\ell}|^{\frac{1}{2}}\dx s\dx t\right)^{\frac{2}{p}}\]
and
\begin{multline}\label{eq.bd-coord}
\mathfrak{Q}_{\mathcal{G},h}(\psi)=\int_{U_{\ell}\times(0,t_{\ell})}\left(\langle(-ih\nabla+\Tilde\A)\Tilde\psi, G^{-1}_{\ell}(-ih\nabla+\Tilde\A)\Tilde\psi  \rangle_{\C^d}+h\Tilde V|\Tilde\psi|^2\right) |G_{\ell}|^{\frac{1}{2}}\dx s\dx t\\
+h^{\frac{3}{2}}\int_{U_{\ell}\times\{t_{\ell}=0\}} \Tilde\gamma |\Tilde\psi|^2|g_{\ell}|^{\frac{1}{2}}\dx s\,,
\end{multline}
where 
\[\Tilde\A=\left(d \Phi_{\ell}\right)^{\mathsf{T}}\circ \A\circ \Phi_{\ell}\,,\qquad \Tilde V=V\circ\Phi_{\ell}\,,\qquad \Tilde\gamma=\gamma\circ\Phi_{\ell}\,,\qquad \Tilde\psi=\psi\circ\Phi_{\ell}\,.\]
Note that in terms one $1$-forms, the first equality means 
\[\Phi_{\ell}^*\left(\sum_{k=1}^d A_{k}\dx x_{k}\right)=\sum_{k=1}^d\Tilde A_{k}\dx s_{k}\,,\]
so that, with a slight abuse of notation, we may write $(\widetilde{\mathcal{G}},\Tilde\psi)=\Phi_{\ell}^*(\mathcal{G},\psi)$. Since the pull-back commutes with the exterior derivative, the magnetic matrix of $\Tilde\B$ is
\[M_{\Tilde\B}=\left(d \Phi_{\ell}\right)^{\mathsf{T}}M_{\B}\left(d \Phi_{\ell}\right)\,,\]
and we may easily deduce the following lemma.
\begin{lemma}\label{lem.pbg}
We let $\widetilde{\mathcal{G}}_{\y_{0}}=(\R^{d-1}\times\R_{+},\mathsf{G}_{\ell}(\y_{0}), V(\x_{0}), \widetilde{\mathcal{A}}_{\y_{0}}^{\mathsf{L}},\gamma(\x_{0}))$. There exists a smooth function $\phi$ on $\R^d$ such that:
\[\left(\x_{0}+(d\Phi_{\ell})_{\y_{0}}\right)^*\mathcal{G}_{\x_{0}}=\widetilde{\mathcal{G}}^\phi_{\y_{0}}\,.\]
\end{lemma}

\subsubsection{Upper bound}
Here is now the estimate related to the boundary case.
\begin{proposition}\label{prop:upperbound}
Let $\x_{0}\in\partial\Omega$. There exists $h_{0}>0, C>0$ such that, for all $h\in(0,h_{0})$,
\[\lambda(\mathcal{G}, h, p)\leq h^{\frac{d}{2}-\frac{d}{p}}h\left(\lambda(\mathcal{G}_{\x_{0}},1,p)+Ch^{\frac{1}{2}}|\log h|\right)\,. \]
\end{proposition}

\begin{proof}
Let us recall that we always have
\[
  \lambda(\mathcal{G}_{\x_{0}},1,p)\leq \lambda(\underline{\mathcal{G}}_{\x_{0}},1,p).
\]
Let us consider first the case of the strict inequality as in Theorem~\ref{theo.0}.
Using the notations of Section~\ref{sec.coord}, we choose
\[\Tilde\psi_{h}(s, t)=h^{-\frac{d}{2p}}\chi_{0}(s,t)e^{i\frac{\varphi_{0}}{h}}\Tilde\Psi_{0}(h^{-1/2}(s-s_{0},t))\,,\]
where $\Tilde\Psi_{0}$ belongs to $\mathcal{S}(\R^{d-1}\times\overline{\R_{+}})$ and $ \chi_{0}$ is a smooth cutoff function such that $0\leq \chi_{0}\leq 1$, being $1$ in $\mathcal{B}(\y_{0},\eps_{0})$ and being zero away from $\mathcal{B}\left(\y_{0},2\eps_{0}\right)$. 
The parameter $\eps_0$ is such that $\supp \Tilde\psi_{h}\subset U_{\ell}\times(0,t_{\ell})$ and $\varphi_0$ satisfies \eqref{eq.Taylor}.
With the same kind of computations as previously (see Proposition~\ref{prop:interior}),  we get
\begin{multline*}
\mathfrak{Q}_{\mathcal{G},h}(\Tilde\psi_{h})\\
\leq hh^{\frac{d}{2}-\frac{d}{p}}\int_{\R^{d-1}\times\R_{+}}\left(\langle(-i\nabla+\Tilde{\mathcal{A}}_{\y_{0}}^{\mathsf{L}})\Tilde\Psi_{0},G_{\ell}(\y_{0})^{-1}(-i\nabla+\Tilde{\mathcal{A}}_{\y_{0}}^{\mathsf{L}})\Tilde\Psi_{0}  \rangle_{\C^d}+\Tilde V(\y_{0})|\Tilde\Psi_{0}|^2\right) \dx \y\\
+hh^{\frac{d}{2}-\frac{d}{p}}\int_{\R^{d-1}\times\{t=0\}} \Tilde\gamma (\y_{0}) |\Tilde\Psi_{0}|^2|g_{\ell}(s_{0})|^{\frac{1}{2}}\dx s+Ch^{\frac{3}{2}}h^{\frac{d}{2}-\frac{d}{p}}\,
\end{multline*}
and
\[
  \|\psi_h\|^p_{\sL^p(\Omega)}=1 + \mathcal{O}(h^{\frac{1}{2}})\,.
\]
Then, it remains to use the change of variable of Lemma~\ref{lem.pbg} and take 
\[\Tilde\Psi_{0}=\left(\x_{0}+(d\Phi_{\ell})_{\y_{0}}\right)^*\Psi_{0}\,,\]
where $\Psi_{0}$ is a minimizer associated with $\lambda(\mathcal{G}_{\x_{0}},1,p)$.

Let us now consider the case when
\[
  \lambda(\mathcal{G}_{\x_{0}},1,p)= \lambda(\underline{\mathcal{G}}_{\x_{0}},1,p)
\]
for which Theorem~\ref{theo.0} does not apply. Let us define $R_h = h^{1/2}|\log h|$. We choose
\[\Tilde\psi_{h}(s, t)=h^{-\frac{d}{2p}}\chi_{0}(R_h^{-1}(s-s_{0},t-2R_h))e^{i\frac{\varphi_{0}}{h}}\Tilde\Psi_{0}(h^{-1/2}(s-s_{0},t-2R_h))\,,\]
where $\Tilde\Psi_{0}=\left(\x_{0}+(d\Phi_{\ell})_{\y_{0}}\right)^*\Psi_{0}\,$
for $\Psi_{0}$ a minimizer associated with $\lambda(\underline{\mathcal{G}}_{\x_{0}},1,p)$, $ \chi_{0}$ is a smooth cutoff function such that $0\leq \chi_{0}\leq 1$, being $1$ in $\mathcal{B}(0,\eps_{0})$ and being zero away from $\mathcal{B}\left(0,2\eps_{0}\right)$. 

We get 
\begin{align*}
\|\psi_h\|^p_{\sL^p(\Omega)} &= \int_{U_{\ell}\times(0,t_{\ell})}|\Tilde\psi|^p|G_{\ell}|^{\frac{1}{2}}\dx s\dx t \\
  &\geq \int_{\mathcal B(0,\eps_0R_hh^{-{1/2}})}|\Tilde\psi_0|^p|G_{\ell}(s_0+\Tilde s h^{1/2},2R_h + \Tilde t h^{1/2} )|^{\frac{1}{2}}\dx \Tilde s\dx \Tilde t\\
  &\geq \int_{\mathcal B(0,\eps_0R_hh^{-{1/2}})}|\Tilde\psi_0|^p|G_{\ell}(s_0,0)|^{\frac{1}{2}}\dx \Tilde s\dx \Tilde t - CR_h\\
  &\geq 1- CR_h -\int_{B(0,\eps_0R_hh^{-{1/2}})^c}|G_{\ell}(s_0,0)|^{\frac{1}{2}}|\Tilde\psi_0|^p\dx \Tilde s\dx \Tilde t\\
  &\geq 1- C(R_h - \exp(-\alpha p\eps_0R_h/h^{\frac{1}{2}})),
\end{align*}
where the last inequality follows from Proposition~\ref{prop.exp}. Hence, by choosing $\eps_0$ such that $\alpha p\eps_{0}=\frac{1}{2}$, we have
\[
  \|\psi_h\|^p_{\sL^p(\Omega)}\geq 1- Ch^{\frac{1}{2}}|\log h|.
\]
 We also obtain
\begin{multline*}
\mathfrak{Q}_{\mathcal{G},h}(\Tilde\psi_{h})\\
\leq hh^{\frac{d}{2}-\frac{d}{p}}\int_{\R^{d-1}\times\R_{+}}\left(\langle(-i\nabla+\Tilde{\mathcal{A}}_{\y_{0}}^{\mathsf{L}})\Tilde\Psi_{0},G_{\ell}(\y_{0})^{-1}(-i\nabla+\Tilde{\mathcal{A}}_{\y_{0}}^{\mathsf{L}})\Tilde\Psi_{0}  \rangle_{\C^d}+\Tilde V(\y_{0})|\Tilde\Psi_{0}|^2\right) \dx \y\\
+hh^{\frac{d}{2}-\frac{d}{p}}\int_{\R^{d-1}\times\{t=0\}} \Tilde\gamma (\y_{0}) |\Tilde\Psi_{0}|^2|g_{\ell}(s_{0})|^{\frac{1}{2}}\dx s+C|\log h|h^{\frac{3}{2}}h^{\frac{d}{2}-\frac{d}{p}}\,
\end{multline*}
and the result follows.
\end{proof}

\section{Lower bound of $\lambda(\mathcal{G}, h, p)$}\label{sec.4}
This section is devoted to the proof of the lower bound in Theorem~\ref{theo.1}.
\subsection{The two-scale localization formula with sliding centers}\label{sec.sliding}
We will need the following type of partition of the unity (see \cite{FR15} for the proof) in the case when $p>2$.
\begin{lemma}\label{lemma-partition}
Let us consider $E=\{(\alpha,\rho,h,\kb)\in(\R_{+})^3\times\Z^d : \alpha\geq \rho\}$. There exists a family of smooth cutoff functions $(\chi^{[\kb]}_{\alpha,\rho,h})_{(\alpha,\rho,h,\kb)\in E}$ on $\R^d$,
with 
\begin{align}
\chi^{[\kb]}_{\alpha,\rho,h}(\x) = \chi^{[0]}_{\alpha,\rho,h}(\x-(2h^{\rho}+h^{\alpha})\kb),
\end{align}
such that $0\leq \chi^{[\kb]}_{\alpha,\rho,h}\leq 1$, 
\begin{align*}
\chi^{[\kb]}_{\alpha,\rho,h} (\x)&=1, &&\text{  on } |\x-(2h^{\rho}+h^{\alpha})\kb|_{\infty}\leq h^{\rho}\,,\\
\chi_{\alpha,\rho,h}^{[\kb]}(\x)&=0, &&\text{ on } |\x-(2h^{\rho}+h^{\alpha})\kb|_{\infty}\geq h^{\rho}+h^{\alpha}\,,
\end{align*} and such that
\begin{equation}\label{eq.part-quad}
\sum_{\kb\in\Z^d}\left(\chi_{\alpha,\rho,h}^{[\kb]}\right)^2=1\,.
\end{equation}
There exists also $D>0$ such that, for all $h>0$,
\begin{equation}\label{partition-remainder}
\sum_{\kb\in\Z^d} |\nabla\chi_{\alpha,\rho,h}^{[\kb]}|^2\leq Dh^{-2\alpha}\,,
\end{equation}
and
\begin{equation}\label{eq.prop-part}
\int_{\R^d} \vert \nabla \chi_{\alpha,\rho,h}^{[\kb]}(\y)\vert^2\,\dx\y\leq Dh^{\rho d}h^{-\alpha-\rho}\,.
\end{equation}
\end{lemma}
The following lemma states that, up to a translation of our quadratic two-scale partition, we may always estimate the global $\sL^p$-norm (resp. the global energy) by the local $\sL^p$-norms (resp. the local energies). It is a generalization and strengthening of \cite[Lemma 4.3]{FR15}.
\begin{lemma}\label{lem.translation}
Let $p \geq 2$.
Let us consider the partition of unity $(\chi_{\alpha,\rho,h}^{[\kb]})$ defined in Lemma~\ref{lemma-partition}, with $\alpha\geq \rho>0$. There exist $C>0$ and $h_{0}>0$ such that for all $\psi\in \sL^p(\Omega)$ and $h\in(0,h_{0})$, there exists $\tau_{\alpha,\rho,h,\psi}=\tau\in\R^d$ such that
\begin{gather*}
\sum_{\kb\in\Z^d} \int_{\Omega} |\Tilde\chi_{\alpha,\rho,h}^{[\kb]}\psi(\x)|^p \dx \x
  \leq\int_{\Omega}|\psi(\x)|^p\dx \x
  \leq (1+Ch^{\alpha-\rho})\sum_{\kb\in\Z^d} \int_{\Omega} |\Tilde\chi_{\alpha,\rho,h}^{[\kb]}\psi(\x)|^p \dx \x,\\
\sum_{\kb\in\Z^d}\mathfrak{Q}_{\mathcal{G}, h}(\Tilde\chi_{\alpha,\rho,h}^{[\kb]}\psi) -
\Tilde Dh^{2-\rho - \alpha}\|\psi\|_{\sL^2(\Omega)}^2 \leq \mathfrak{Q}_{\mathcal{G}, h}(\psi) \leq \sum_{\kb\in\Z^d}\mathfrak{Q}_{\mathcal{G}, h}(\Tilde\chi_{\alpha,\rho,h}^{[\kb]}\psi)\,,
\end{gather*}
with $\Tilde\chi_{\alpha,\rho,h}^{[\kb]}(\x)=\chi_{\alpha,\rho,h}^{[\kb]}(\x-\tau)$.
Moreover, the translated partition $(\Tilde\chi_{\alpha,\rho,h}^{[\kb]})$ still satisfies \eqref{partition-remainder}.
\end{lemma}

\begin{proof}
Since $\displaystyle{\sum_{\kb\in\Z^d} (\Tilde\chi_{\alpha,\rho,h}^{[\kb]})^2 =1}$ for all $\tau$, and $p\geq 2$, 
we have immediately
\[
 \sum_{\kb\in\Z^d} \int_{\Omega} |\Tilde\chi_{\alpha,\rho,h}^{[\kb]}\psi(\x)|^p \dx \x
  \leq\int_{\Omega}|\psi(\x)|^p\dx \x\;,
\]
Notice that, by paving $\R^d$ and by using that $\chi_{\alpha,\rho,h}^{[0]}=1$ on a box of sidelength $2h^{\rho}$, 
\begin{align}
\int_{[0,2h^{\rho}+h^{\alpha})^d}\sum_{\kb\in\Z^d} \left|\chi_{\alpha,\rho,h}^{[0]}(\x - (2h^{\rho}+h^{\alpha})\kb-\tau)\right|^p\,\dx\tau = \int_{\R^d} \left|\chi_{\alpha,\rho,h}^{[0]}(\x-\y)\right|^p\,\dx\y
\geq 2^dh^{d\rho}.
\end{align}
Therefore, by changing the order of the integrations,
\begin{align*}
\frac{1}{(2h^{\rho}+h^{\alpha})^d} \int_{[0,2h^{\rho}+h^{\alpha})^d}\left(
\sum_{\kb\in\Z^d} \int_{\Omega} |\Tilde\chi_{\alpha,\rho,h}^{[\kb]}\psi(\x)|^p \dx \x\right)
\,\dx\tau \geq \frac{2^d h^{d\rho}}{(2h^{\rho}+h^{\alpha})^d}\int_{\Omega}|\psi(\x)|^p\dx \x,
\end{align*}
and thus 
\begin{multline*}
\frac{1}{(2h^{\rho}+h^{\alpha})^d} \int_{[0,2h^{\rho}+h^{\alpha})^d}\left(\int_{\Omega}|\psi(\x)|^p\dx \x - 
\sum_{\kb\in\Z^d} \int_{\Omega} |\Tilde\chi_{\alpha,\rho,h}^{[\kb]}\psi(\x)|^p \dx \x\right)
\,\dx\tau \\
\leq \Bigl(1 - \frac{ 2^d h^{d\rho}}{(2h^{\rho}+h^{\alpha})^d}\Bigr)\int_{\Omega}|\psi(\x)|^p\dx \x\,.
\end{multline*}
This last inequality is in the form 
\[\frac{1}{L^d}\int_{[0,L]^d} f(\tau) \dx\tau\leq A\,,\]
with a non-negative and integrable function $f$, so that we get
\[\left|\{\tau\in[0,L]^d : f(\tau)\leq 3A\}\right|\geq \frac{2L^d}{3}\,.\]
In our particular situation, we deduce that the set of \(\tau \in [0,2h^{\rho}+h^{\alpha})^d\) such that 
\begin{equation}\label{eq.Lp-loc-Lp}
\sum_{\kb\in\Z^d} \int_{\Omega} |\Tilde\chi_{\alpha,\rho,h}^{[\kb]}\psi(\x)|^p \dx \x
\ge 1 -3 \Bigl(1 - \frac{ 2^d h^{d\rho}}{(2h^{\rho}+h^{\alpha})^d}\Bigr)
\int_{\Omega}|\psi(\x)|^p\dx \x,
\end{equation}
has measure at least \(\frac{2}{3} (2h^{\rho}+h^{\alpha})^d\). We may notice here that 
\[\frac{ 2^d h^{d\rho}}{(2h^{\rho}+h^{\alpha})^d}=1+\mathcal{O}(h^{\alpha-\rho})\,.\]

With the localization formula associated with the partition 
of unity $(\Tilde\chi_{\alpha,\rho,h}^{[\kb]})$ that is adapted to $\psi$, we infer
\[
\mathfrak{Q}_{\mathcal{G}, h}(\psi)
=\sum_{\kb\in\Z^d}\mathfrak{Q}_{\mathcal{G}, h}(\Tilde\chi_{\alpha,\rho,h}^{[\kb]}\psi) 
-h^2\sum_{\kb\in\Z^d}\|\nabla\Tilde\chi_{\alpha,\rho,h}^{[\kb]}\psi\|_{\sL^2(\Omega)}^2\,,
\]
So that in particular
\[
 \mathfrak{Q}_{\mathcal{G}, h}(\psi) 
 \le \sum_{\kb\in\Z^d}\mathfrak{Q}_{\mathcal{G}, h}(\Tilde\chi_{\alpha,\rho,h}^{[\kb]}\psi).
\]
With Fubini's theorem and \eqref{eq.prop-part}, we also observe that, for all $\x\in\R^d$,
\begin{multline*}
 \frac{1}{(2h^{\rho}+h^{\alpha})^d}\int_{[0,2h^{\rho}+h^{\alpha})^d}\sum_{\kb\in\mathbb{Z}} \vert \nabla \chi_{\alpha,\rho,h}^{[0]}(\x - (2h^{\rho}+h^{\alpha})\kb-\tau)\vert^2\,\dx\tau \\
 = \frac{1}{(2h^{\rho}+h^{\alpha})^d} \int_{\R^d} \vert \nabla \chi_{\alpha,\rho,h}^{[0]}(\x-\y)\vert^2\,\dx\y
\leq \frac{C}{h^{\rho + \alpha}}.
\end{multline*}
Therefore,
\begin{multline*}
\frac{1}{(2h^{\rho}+h^{\alpha})^d} \int_{[0,2h^{\rho}+h^{\alpha})^d}\Bigl(
\sum_{\kb\in\Z^d}\mathfrak{Q}_{\mathcal{G}, h}(\Tilde\chi_{\alpha,\rho,h}^{[\kb]}\psi)
- \mathfrak{Q}_{\mathcal{G}, h}(\psi)\Bigr) \\
= \frac{h^2}{(2h^{\rho}+h^{\alpha})^d} \Bigl(\int_{[0,2h^{\rho}+h^{\alpha})^d}
\sum_{\kb\in\Z^d}\|\nabla\Tilde\chi_{\alpha,\rho,h}^{[\kb]}\psi\|_{\sL^2(\Omega)}^2\Bigr)\dx \tau\,,\\
\le C h^{2 - \rho- \alpha} \|\psi\|_{\sL^2(\Omega)}^2
\end{multline*}
And thus the set of \(\tau \in [0,2h^{\rho}+h^{\alpha})^d\) such that 
\begin{equation*}
\frac{1}{(2h^{\rho}+h^{\alpha})^d} \int_{[0,2h^{\rho}+h^{\alpha})^d}
\sum_{\kb\in\Z^d}\mathfrak{Q}_{\mathcal{G}, h}(\Tilde\chi_{\alpha,\rho,h}^{[\kb]}\psi)
- \mathfrak{Q}_{\mathcal{G}, h}(\psi) \le 3C h^{2 - \rho- \alpha} \|\psi\|_{\sL^2(\Omega)}^2
\end{equation*}
has measure at least \(\frac{2}{3} (2h^{\rho}+h^{\alpha})^d\).

We conclude that the desired estimates are satisfied for a set of \(\tau\) of measure at least  \(\frac{1}{3} (2h^{\rho}+h^{\alpha})^d\).
\end{proof}

\begin{remark}
Note that if $p=2$, we choose $\alpha=\rho$ and we do not need Lemma~\ref{lem.translation}.
\end{remark}

\subsection{Approximation by the homogeneous geometry}
In this section, we prove Proposition~\ref{prop.minorationL2} and the lower bound in Theorem~\ref{theo.1}. Since the lower bound in the nonlinear case ($p>2$) is more subtle we will mainly focus on this case. Note that in many places the estimates when $p=2$ are better and easier to obtain. 

Keeping in mind the estimate of the quadratic form of Lemma~\ref{lem.translation}, we must now approximate the local energies $\mathfrak{Q}_{\mathcal{G}, h}(\Tilde\chi_{\alpha,\rho,h}^{[\kb]}\psi)$. To lighten the notation, we let $\psi_{\kb}=\Tilde\chi_{\alpha,\rho,h}^{[\kb]}\psi$.
\subsubsection{Interior estimates}\label{sec.int-esti}
Let us consider the $\kb \in \Z^d$ such that $\supp (\Tilde\chi_{\alpha,\rho,h}^{[\kb]})\cap \partial\Omega=\emptyset$. We have
\[\|(-ih\nabla+\A)\psi_{\kb}\|^2_{\sL^2(\Omega)}=\|(-ih\nabla+\mathcal{A}^{\mathsf{L}}_{\x_{\kb}}+\mathcal{R}_{\kb})\Psi_{\kb}\|^2_{\sL^2(\Omega)}\,,\]
where $\Psi_{\kb}=e^{i\varphi_{\kb}/h}\psi_{\kb}$ for a suitable choice of gauge $\varphi_{\kb}$ and the Taylor remainder $\mathcal{R}_{\kb}$ satisfies, on the support of $\Tilde\chi_{\alpha,\rho,h}^{[\kb]}$, $|\mathcal{R}_{\kb}|\leq Ch^{2\rho}$. An elementary inequality implies, for all $\eps\in(0,1)$,
\[\|(-ih\nabla+\A)\psi_{\kb}\|^2_{\sL^2(\Omega)}\geq (1-\eps)\|(-ih\nabla+\mathcal{A}^{\mathsf{L}}_{\x_{\kb}})\Psi_{\kb}\|^2_{\sL^2(\Omega)}-C^2 h^{4\rho}\eps^{-1}\|\psi_{\kb}\|_{\sL^2(\Omega)}^2\,.\]
Moreover, we have
\[h\int_{\Omega} V(\x)|\psi_{\kb}|^2\dx\x\geq h\int_{\Omega} V(\x_{\kb})|\psi_{\kb}|^2\dx\x-\hat Ch^{1+\rho}\|\psi_{\kb}\|^2_{\sL^2(\Omega)}\,.\]
Since we investigate the case $p>2$, one need to control the remainder involving $\|\psi_{\kb}\|^2_{\sL^2(\Omega)}$ and not $\|\psi_{\kb}\|^2_{\sL^p(\Omega)}$. 
Note that we do not have to care about this when $p=2$.

We have, by homogeneity and the min-max principle,
\[\lambda(\mathcal{G}_{\x_{\kb}}, 1, 2)h\|\Psi_{\kb}\|^2_{\sL^2(\Omega)}\leq \mathfrak{Q}_{\mathcal{G}_{\x_{\kb}}, h}(\Psi_{\kb})\,.\]
By using Assumption~\ref{a.0}, we deduce that
\begin{multline*}
\mathfrak{Q}_{\mathcal{G}, h}(\psi_{\kb})
\geq (1-\eps)\|(-ih\nabla+\mathcal{A}^{\mathsf{L}}_{\x_{\kb}})\Psi_{\kb}\|^2_{\sL^2(\Omega)}+h\int_{\Omega} V(\x_{\kb})|\Psi_{\kb}|^2\dx\x\\
-(\Tilde C\eps^{-1}h^{4\rho-1}+\Tilde C h^{\rho})\mathfrak{Q}_{\mathcal{G}_{\x_{\kb}}, h}(\Psi_{\kb})\,.
\end{multline*}
Then, we get
\[h\int_{\Omega} V(\x_{\kb})|\Psi_{\kb}|^2\dx\x\geq (1-\eps)h\int_{\Omega} V(\x_{\kb})|\Psi_{\kb}|^2\dx\x-\eps h\left(\max_{\overline{\Omega}}|V|\right)\|\Psi_{\kb}\|^2_{\sL^2(\Omega)}\]
Therefore we deduce 
\begin{equation}\label{eq.control-resteQ}
\mathfrak{Q}_{\mathcal{G}, h}(\psi_{\kb})\geq (1-\Tilde C\eps-\Tilde C\eps^{-1}h^{4\rho-1}-\Tilde C h^{\rho})\mathfrak{Q}_{\mathcal{G}_{\x_{\kb}}, h}(\Psi_{\kb})\,.
\end{equation}
We choose $\eps={h^{2\rho-\frac{1}{2}}}$ and we notice that, since $\rho\in(0,\frac{1}{2})$, then we can forget the term in $h^{\rho}$. By definition of the infimum and homogeneity, it follows that
\begin{equation}\label{eq.min-int}
\mathfrak{Q}_{\mathcal{G}, h}(\psi_{\kb})\geq (1-Ch^{2\rho-\frac{1}{2}})\lambda(\mathcal{G}_{\cb_{\kb}}, 1, p)h^{1+\frac{d}{2}-\frac{d}{p}}\|\psi_{\kb}\|^2_{\sL^p(\Omega)}\,.
\end{equation}
We deduce in particular
\begin{equation}\label{eq.min-int'}
\mathfrak{Q}_{\mathcal{G}, h}(\psi_{\kb})\geq (1-Ch^{2\rho-\frac{1}{2}})h^{1+\frac{d}{2}-\frac{d}{p}}\inf_{\x\in\Omega}\lambda(\mathcal{G}_{\x}, 1, p)\|\psi_{\kb}\|^2_{\sL^p(\Omega)}\,.
\end{equation}

\subsubsection{Boundary estimates}
Let us consider the $\kb$ such that $\supp (\Tilde\chi_{\alpha,\rho,h}^{[\kb]})\cap \partial\Omega\neq\emptyset$. Let us consider $\cb_{\kb}\in\supp (\Tilde\chi_{\alpha,\rho,h}^{[\kb]})\cap \partial\Omega\neq\emptyset$. 
On $\supp (\Tilde\chi_{\alpha,\rho,h}^{[\kb]})$ we may consider the local coordinates (near some point $\mathsf{X}_{\ell}$) introduced in Section~\ref{sec.coord} and $\x_{0}=\cb_{\kb}$.
The coordinates of $\cb_{\kb}$ are $(s_{\kb}, t_{\kb})$ in the parametrization $\Phi_{\ell}$. We use the expression \eqref{eq.bd-coord} and the Taylor formula to get
\begin{multline*}
\mathfrak{Q}_{\mathcal{G}, h}(\psi_{\kb})\geq (1-Ch^{\rho})\int_{U_{\ell}\times(0,t_{\ell})} \langle(-ih \nabla+\Tilde \A)\Tilde\psi_{\kb}, G_{\ell}(\cb_{\kb})^{-1} (-ih \nabla+\Tilde \A)\Tilde\psi_{\kb}\rangle_{\R^d}|G_{\ell}(\cb_{k})|^{\frac{1}{2}}\dx\y\\
+h\int_{U_{\ell}\times(0,t_{\ell})} V(\cb_{\kb})|\Tilde\psi_{\kb}|^2|G_{\ell}(\cb_{k})|^{\frac{1}{2}}\dx\y-\gamma(\cb_{\kb})h^{\frac{3}{2}}\int_{U_{\ell}\times\{0\}}|\Tilde\psi_{\kb}|^2|g_{\ell}(s_{\kb})|^{\frac{1}{2}}\dx s\\
-Ch^{\frac{3}{2}+\rho}\int_{U_{\ell}\times\{0\}}|\Tilde\psi_{\kb}|^2\dx s-Ch^{1+\rho}\int_{U_{\ell}\times(0,t_{\ell})}|\Tilde\psi_{\kb}|^2\dx\y\,.
\end{multline*}
By using Lemma~\ref{lem.trace-eps} with $\eps=h^{\frac{1}{2}}$, we deduce that
\[\int_{U_{\ell}\times\{0\}}|\Tilde\psi_{\kb}|^2\dx s\leq Ch^{\frac{1}{2}}\|\nabla|\Tilde\psi_{\kb}|\|^2+Ch^{-\frac{1}{2}}\|\Tilde\psi_{\kb}\|^2\,,\]
so that, with the diamagnetic inequality,
\[\int_{U_{\ell}\times\{0\}}|\Tilde\psi_{\kb}|^2\dx s\leq Ch^{-\frac{3}{2}}\|(-ih\nabla+\Tilde\A)\Tilde\psi_{\kb}\|^2+Ch^{-\frac{1}{2}}\|\Tilde\psi_{\kb}\|^2\,,\]
and thus
\begin{multline*}
\mathfrak{Q}_{\mathcal{G}, h}(\psi_{\kb})\geq (1-Ch^{\rho})\int_{U_{\ell}\times(0,t_{\ell})} \langle(-ih \nabla+\Tilde \A)\Tilde\psi_{\kb}, G_{\ell}(\cb_{\kb})^{-1} (-ih \nabla+\Tilde \A)\Tilde\psi_{\kb}\rangle_{\R^d}|G_{\ell}(\cb_{\kb})|^{\frac{1}{2}}\dx\y\\
+h\int_{U_{\ell}\times(0,t_{\ell})} V(\cb_{\kb})|\Tilde\psi_{\kb}|^2|G_{\ell}(\cb_{k})|^{\frac{1}{2}}\dx\y-\gamma(\cb_{\kb})h^{\frac{3}{2}}\int_{U_{\ell}\times\{0\}}|\Tilde\psi_{\kb}|^2|g_{\ell}(s_{\kb})|^{\frac{1}{2}}\dx s\\
-Ch^{1+\rho}\int_{U_{\ell}\times(0,t_{\ell})}|\Tilde\psi_{\kb}|^2\dx\y\,.
\end{multline*}
Now, we approximate the vector potential as in Section~\ref{sec.int-esti} and we get 
\begin{equation}\label{eq.Qloc-bnd}
\mathfrak{Q}_{\mathcal{G}, h}(\psi_{\kb})\geq (1-Ch^{2\rho-\frac{1}{2}})\mathfrak{Q}_{\mathcal{G}^*_{\cb_{\kb}}, h}(\Tilde\psi_{\kb})\,,
\end{equation}
and thus
\begin{equation}\label{eq.min-bd}
\mathfrak{Q}_{\mathcal{G}, h}(\psi_{\kb})\geq (1-Ch^{2\rho-\frac{1}{2}})\lambda(\mathcal{G}_{\cb_{\kb}}, 1, p)h^{1+\frac{d}{2}-\frac{d}{p}}\|\psi_{\kb}\|^2_{\sL^p(\Omega)}\,.
\end{equation}
We get
\begin{equation}\label{eq.min-bd'}
\mathfrak{Q}_{\mathcal{G}, h}(\psi_{\kb})\geq (1-Ch^{2\rho-\frac{1}{2}})h^{1+\frac{d}{2}-\frac{d}{p}}\inf_{\x\in\partial\Omega}\lambda(\mathcal{G}_{\x}, 1, p)\|\psi_{\kb}\|^2_{\sL^p(\Omega)}\,.
\end{equation}
From \eqref{eq.min-int'} and \eqref{eq.min-bd'}, we infer that there exist $h_{0}>0$, $C>0$ such that, for all $\kb\in\Z^d$, 
\begin{equation}\label{eq.local-energy-lb}
\mathfrak{Q}_{\mathcal{G}, h}(\psi_{\kb})\geq (1-Ch^{2\rho-\frac{1}{2}})h^{1+\frac{d}{2}-\frac{d}{p}}\inf_{\x\in\overline{\Omega}}\lambda(\mathcal{G}_{\x}, 1, p)\|\psi_{\kb}\|^2_{\sL^p(\Omega)}\,.
\end{equation}
We want to add the different local contributions. We use the following estimate
\[h^{2-\rho-\alpha}\|\psi\|^2_{\sL^2(\Omega)}=h^{2-\rho-\alpha}\sum_{\kb\in\Z^d}\|\psi_{\kb}\|^2_{\sL^2(\Omega)}\leq Ch^{1-\alpha-\rho}\sum_{\kb\in\Z^d}\mathfrak{Q}_{\mathcal{G}, h}(\psi_{\kb}) \,,\]
where we used \eqref{eq.min-int}, \eqref{eq.Qloc-bnd} and the positivity in Assumption~\ref{a.0}. Using Lemma~\ref{lem.translation}, we get
\[h^{2-\rho-\alpha}\|\psi\|^2_{\sL^2(\Omega)}\leq Ch^{1-\alpha-\rho}\mathfrak{Q}_{\mathcal{G}, h}(\psi)\,.\]
Summing over $\kb$ in \eqref{eq.local-energy-lb} and using Lemma~\ref{lem.translation} to reconstruct the total $\sL^p$-norm from the local ones, we get 
\[\mathfrak{Q}_{\mathcal{G}, h}(\psi)\geq  (1-Ch^{1-\alpha - \rho})(1-Ch^{2\rho-\frac{1}{2}})(1-Ch^{\alpha-\rho})h^{1+\frac{d}{2}-\frac{d}{p}}\inf_{\x\in\overline{\Omega}}\lambda(\mathcal{G}_{\x}, 1, p)\|\psi\|^2_{\sL^p(\Omega)}\,.\]
We choose 
\[1-\alpha - \rho=2\rho-\frac{1}{2}=\alpha-\rho\,,\text{ or equivalently }\quad \rho=\frac{1}{3}\,,\quad\alpha=\frac{1}{2}\,.\]
This gives the lower bound in Theorem~\ref{theo.1}.

\begin{remark}
In the case $p=2$, we are led to the choice $1-2\rho=2\rho-\frac{1}{2}$ and thus $\rho=\frac{3}{8}$. Note that, in this case, the term $1-Ch^{\alpha-\rho}$ is replaced by $1$. From this we find the remainder of order $\mathcal{O}(h^{\frac{5}{4}})$ given in Proposition~\ref{prop.minorationL2}.
\end{remark}

\section{Semiclassical localization}\label{sec.5}
In relation with the estimates of Section~\ref{sec.4} and using the upper bound in Theorem~\ref{theo.1}, we may deduce that the minimizers concentrate near the minima of the concentration function $\overline{\Omega}\ni\x\mapsto \lambda(\mathcal{G}_{\x}, 1,p)$.
\subsection{A rough localization estimate}
Before obtaining the exponential localization, we start by proving a weaker result.
\begin{proposition}\label{prop.Lp-loc}
For all $\eps>0$, there exist $h_{0}, C>0$ such that for all $h\in(0, h_{0})$ and all $\sL^p$-normalized minimizer $\psi_{h}$ of \eqref{eq.Sobo}, we have
\[\|\psi_{h}\|_{\sL^p(\complement \mathcal{M}_{\eps})}\leq Ch^{\frac{1}{6p}}\,,\]
where $ \mathcal{M}_{\eps}$ is defined in \eqref{eq.Me}.
\end{proposition}
\begin{proof}
We use again Lemma~\ref{lem.translation} with $\psi=\psi_{h}$ and $\rho=\frac{1}{3}\,,\alpha=\frac{1}{2}$ to get
\[\sum_{\kb\in\Z^d}\mathfrak{Q}_{\mathcal{G}, h}(\Tilde\chi_{\alpha,\rho,h}^{[\kb]}\psi) -
\Tilde Dh^{2-\rho - \alpha}\|\psi\|_{\sL^2(\Omega)}^2 \leq \lambda(\mathcal{G}, h, p)\|\psi_{h}\|^2_{\sL^p(\Omega)}\,.\]
We deduce that
\[(1-Ch^{\frac{1}{6}})\sum_{\kb\in\Z^d}\mathfrak{Q}_{\mathcal{G}, h}(\Tilde\chi_{\alpha,\rho,h}^{[\kb]}\psi_{h}) \leq \lambda(\mathcal{G}, h, p)\|\psi_{h}\|^2_{\sL^p(\Omega)}\,.\]
Therefore, we find
\begin{equation*}
\sum_{\kb\in\Z^d}\mathfrak{Q}_{\mathcal{G}, h}(\psi_{h, \kb}) \leq  \lambda(\mathcal{G}, h, p)\|\psi_{h}\|^2_{\sL^p(\Omega)}+ Ch^{\frac{1}{6}}h^{1+\frac{d}{2}-\frac{d}{p}}\|\psi_{h}\|^2_{\sL^p(\Omega)}\,.
\end{equation*}
We again have by Lemma~\ref{lem.translation},
\[\|\psi_{h}\|^2_{\sL^p(\Omega)}=\left(\int_{\Omega} |\psi_{h}|^p\dx\x\right)^{\frac{2}{p}}\leq (1+Ch^{\frac{1}{6}})\left(\sum_{\kb\in\Z^d} \|\Tilde\chi_{\alpha,\rho,h}^{[\kb]}\psi_{h}\|^p_{\sL^p(\Omega)}\right)^{\frac{2}{p}} \,,\]
and thus
\[\|\psi_{h}\|^2_{\sL^p(\Omega)}\leq\left(\sum_{\kb\in\Z^d} \|\Tilde\chi_{\alpha,\rho,h}^{[\kb]}\psi_{h}\|^p_{\sL^p(\Omega)}\right)^{\frac{2}{p}}+Ch^{\frac{1}{6}}\|\psi_{h}\|^2_{\sL^p(\Omega)}\,.\]
By using that $p\geq 2$, we deduce that
\begin{equation}\label{eq.sumQ}
\sum_{\kb\in\Z^d}\mathfrak{Q}_{\mathcal{G}, h}(\psi_{h, \kb}) \leq  \lambda(\mathcal{G}, h, p)\sum_{\kb\in\Z^d} \|\psi_{h,\kb}\|^2_{\sL^p(\Omega)}+ Ch^{\frac{1}{6}}h^{1+\frac{d}{2}-\frac{d}{p}}\|\psi_{h}\|^2_{\sL^p(\Omega)}\,.
\end{equation}
Then, we consider the local energies. For $\kb$ such that $\supp (\Tilde\chi_{\alpha,\rho,h}^{[\kb]})\cap \partial\Omega=\emptyset$, we have
\begin{equation*}
\mathfrak{Q}_{\mathcal{G}, h}(\psi_{h, \kb})\geq (1-Ch^{\frac{1}{6}})\lambda(\mathcal{G}_{\x_{\kb}}, 1, p)h^{1+\frac{d}{2}-\frac{d}{p}}\|\psi_{h, \kb}\|^2_{\sL^p(\Omega)}\,,
\end{equation*}
and, for $\kb$ such that $\supp (\Tilde\chi_{\alpha,\rho,h}^{[\kb]})\cap \partial\Omega\neq\emptyset$,
\begin{equation*}
\mathfrak{Q}_{\mathcal{G}, h}(\psi_{h, \kb})\geq (1-Ch^{\frac{1}{6}})\lambda(\mathcal{G}_{\cb_{\kb}}, 1, p)h^{1+\frac{d}{2}-\frac{d}{p}}\|\psi_{h, \kb}\|^2_{\sL^p(\Omega)}\,.
\end{equation*}
For $\eps, h>0$, let us introduce
\[K_{\eps, h}=\left\{\kb\in\Z^d : \supp (\Tilde\chi_{\alpha,\rho,h}^{[\kb]})\cap\mathcal{M}_{\frac{\eps}{2}}\neq\emptyset\right\}\,.\]
For all $\eps>0$, there exists $c_{\eps}>0$ such that for all $\kb\notin K_{\eps, h}$, we have
\begin{equation}\label{eq.large-lb}
\mathfrak{Q}_{\mathcal{G}, h}(\psi_{h, \kb})\geq (1-Ch^{\frac{1}{6}})\left(\inf_{\x\in\overline{\Omega}}\lambda(\mathcal{G}_{\x}, 1, p)+c_{\eps}\right)h^{1+\frac{d}{2}-\frac{d}{p}}\|\psi_{h,\kb}\|^2_{\sL^p(\Omega)}\,.
\end{equation}
From \eqref{eq.sumQ}, we get
\begin{multline*}
\sum_{\kb\notin K_{\eps,h}}\left(\mathfrak{Q}_{\mathcal{G}, h}(\psi_{h, \kb})- \lambda(\mathcal{G}, h, p)\|\psi_{h,\kb}\|^2_{\sL^p(\Omega)}\right)+\sum_{\kb\in K_{\eps,h}}\left(\mathfrak{Q}_{\mathcal{G}, h}(\psi_{h, \kb})- \lambda(\mathcal{G}, h, p)\|\psi_{h,\kb}\|^2_{\sL^p(\Omega)}\right)\\
\leq Ch^{\frac{1}{6}}h^{1+\frac{d}{2}-\frac{d}{p}}\|\psi_{h}\|^2_{\sL^p(\Omega)}  \,.
\end{multline*}
On one hand, from \eqref{eq.large-lb} and the upper bound on $\lambda(\mathcal{G}, h, p)$, we find the existence of  $\Tilde c_{\eps}>0$ such that, for all $h\in (0, h_{0})$,
\begin{equation*}
\sum_{\kb\notin K_{\eps,h}}\left(\mathfrak{Q}_{\mathcal{G}, h}(\psi_{h, \kb})-\lambda(\mathcal{G}, h, p)\|\psi_{h, \kb}\|^2_{\sL^p(\Omega)}\right)\\
\geq \Tilde c_{\eps}h^{1+\frac{d}{2}-\frac{d}{p}}\sum_{\kb\notin K_{\eps, h}}\|\psi_{h, \kb}\|^2_{\sL^p(\Omega)}\,.
\end{equation*}
On the other hand, we get, by using the definition of $\lambda(\mathcal{G}, h, p)$,
\[\sum_{\kb\in K_{\eps,h}}\left(\mathfrak{Q}_{\mathcal{G}, h}(\psi_{h, \kb})- \lambda(\mathcal{G}, h, p)\|\psi_{h,\kb}\|^2_{\sL^p(\Omega)}\right)\geq 0\,.\]
It follows that
\[\Tilde c_{\eps}\sum_{\kb\notin K_{\eps, h}}\|\psi_{h, \kb}\|^2_{\sL^p(\Omega)}\leq Ch^{\frac{1}{6}}\|\psi_{h}\|^2_{\sL^p(\Omega)} \,.\]
Since $p\geq 2$, we deduce that
\[\sum_{\kb\notin K_{\eps, h}}\int_{\Omega}|\Tilde\chi_{\alpha,\rho,h}^{[\kb]}|^p|\psi_{h}|^p\dx\x\leq Ch^{\frac{p}{12}}\|\psi_{h}\|^p_{\sL^p(\Omega)} \,.\]
Thus we have
\begin{multline*}
\sum_{\kb\notin K_{\eps, h}}\int_{\Omega}|\Tilde\chi_{\alpha,\rho,h}^{[\kb]}|^2|\psi_{h}|^p\dx\x
-\sum_{\kb\notin K_{\eps, h}}\int_{\Omega}\left(|\Tilde\chi_{\alpha,\rho,h}^{[\kb]}|^2-|\Tilde\chi_{\alpha,\rho,h}^{[\kb]}|^p\right)|\psi_{h}|^p\dx\x\\
\leq Ch^{\frac{p}{12}}\|\psi_{h}\|^p_{\sL^p(\Omega)} \,.
\end{multline*}
With \eqref{eq.Lp-loc-Lp}, we infer
\begin{multline*}
\int_{\complement\mathcal{M}_{\eps}}|\psi_{h}|^p\dx\x=\sum_{\kb\notin K_{\eps, h}}\int_{\complement\mathcal{M}_{\eps}}|\Tilde\chi_{\alpha,\rho,h}^{[\kb]}|^2|\psi_{h}|^p\dx\x\ \leq\sum_{\kb\notin K_{\eps, h}}\int_{\Omega}|\Tilde\chi_{\alpha,\rho,h}^{[\kb]}|^2|\psi_{h}|^p\dx\x\\
\leq C(h^{\frac{p}{12}}+h^{\frac{1}{6}})\|\psi_{h}\|^p_{\sL^p(\Omega)}\,.
\end{multline*}
\end{proof}

\subsection{Application to the exponential estimates}
Now, Proposition~\ref{prop.Lp-loc} gives an a priori control of the nonlinear potential $-\lambda(\mathcal{G}, h, p)|\psi_{h}|^{p-2}$ away from the minimal set $\mathcal{M}$. Therefore, in this region, we are essentially reduced to a perturbation of a linear equation and we may establish decay estimates \emph{\`a la} Agmon.
 \begin{proposition}\label{prop.exp-loc}
For all $\eps>0$, $\rho\in\left(0, \frac{1}{2}\right)$, there exists $h_{0}, C>0$ such that for all $h\in(0, h_{0})$ and all $\sL^p$-normalized minimizer $\psi_{h}$ of \eqref{eq.Sobo}, we have
\[\|\psi_{h}\|_{\sL^p(\complement \mathcal{M}_{2\eps})}\leq Ce^{-\eps h^{-\rho}}\|\psi_{h}\|_{\sL^p(\Omega)}\,.\]
\end{proposition}
\begin{proof}
Let us first consider a function $v$ in the form domain of $\mathfrak{Q}_{\mathcal{G}, h}$ and supported away from $\mathcal{M}_{\eps}$. We have, with the H\"older inequality,
\[\lambda(\mathcal{G}, h, p)\int_{\Omega}|\psi_{h}|^{p-2}|v|^2\dx\x\leq \lambda(\mathcal{G}, h, p)\|v\|^2_{\sL^p(\Omega)}\left(\int_{\complement\mathcal{M}_{\eps}}|\psi_{h}|^p\dx\x\right)^{\frac{p-2}{p}}\,.\]
Thus, by using Proposition~\ref{prop.Lp-loc}, we get
\begin{equation}\label{eq.controlNL}
\lambda(\mathcal{G}, h, p)\int_{\Omega}|\psi_{h}|^{p-2}|v|^2\dx\x\leq Ch^{\frac{p-2}{6p}}\lambda(\mathcal{G}, h, p)\|v\|^2_{\sL^p(\Omega)}\leq Ch^{\frac{p-2}{6p}}\mathfrak{Q}_{\mathcal{G}, h}(v)\,.
\end{equation}
In other words, the nonlinear potential is a perturbation of the linear part in the sense of quadratic forms. Note that the equation satisfied by $\psi_{h}$ reads
\[\mathfrak{L}^\mathsf{NL}_{\mathcal{G}, h}\psi_{h}=\left((-ih\nabla+\A)^2+hV-\lambda(\mathcal{G}, h, p)|\psi_{h}|^{p-2}\right)\psi_{h}=0\,,\]
with the Robin condition
\[(-ih\nabla+\A)\psi\cdot\n(\x)=-ih^{\frac{1}{2}}\gamma(\x)\psi(\x),\quad\x\in\partial\Omega\,.\]
Let us now introduce our exponential weight. We consider a smooth cutoff function $0\leq\chi\leq 1$ supported away from $\mathcal{M}_{\eps}$ and being $1$ away from $\mathcal{M}_{2\eps}$. With the localization formula, we get
\[\mathfrak{Q}^{\mathsf{NL}}_{\mathcal{G}, h}(e^{\chi h^{-\rho}\dist(\x,\mathcal{M})}\psi_{h})-Ch^{2-2\rho}\|e^{\chi h^{-\rho}\dist(\x,\mathcal{M})}\psi_{h}\|^2_{\sL^2(\Omega)}\leq 0\,;
\]
all the norms are finite in view of the boundedness of the domain \(\Omega\).
Now we want to distinguish between the region close to $\mathcal{M}$ and far from $\mathcal{M}$. Thus we introduce a smooth quadratic partition of the unity $\chi_{1}^2+\chi_{2}^2=1$ such that $\supp(\chi_{2})\subset\complement\mathcal{M}_{\eps}$. Using again the localization formula, we deduce that
\begin{multline}\label{eq.loc.exp}
\mathfrak{Q}^{\mathsf{NL}}_{\mathcal{G}, h}(\chi_{1}e^{\chi h^{-\rho}\dist(\x,\mathcal{M})}\psi_{h})+\mathfrak{Q}^{\mathsf{NL}}_{\mathcal{G}, h}(\chi_{2}e^{\chi h^{-\rho}\dist(\x,\mathcal{M})}\psi_{h})\\
-\Tilde Ch^{2-2\rho}\|\chi_{1}e^{\chi h^{-\rho}\dist(\x,\mathcal{M})}\psi_{h}\|^2_{\sL^2(\Omega)}-\Tilde Ch^{2-2\rho}\|\chi_{2}e^{\chi h^{-\rho}\dist(\x,\mathcal{M})}\psi_{h}\|^2_{\sL^2(\Omega)}\leq 0\,.
\end{multline}
On one hand, by \eqref{eq.controlNL}, we have
\begin{align}\label{eq.chi2}
&\mathfrak{Q}^{\mathsf{NL}}_{\mathcal{G}, h}(\chi_{2}e^{\chi h^{-\rho}\dist(\x,\mathcal{M})}\psi_{h})-\Tilde Ch^{2-2\rho}\|\chi_{2}e^{\chi h^{-\rho}\dist(\x,\mathcal{M})}\psi_{h}\|^2_{\sL^2(\Omega)}\\
\nonumber&\geq (1-Ch^{\frac{p-2}{6p}})\mathfrak{Q}_{\mathcal{G}, h}(\chi_{2}e^{\chi h^{-\rho}\dist(\x,\mathcal{M})}\psi_{h})-\Tilde Ch^{2-2\rho}\|\chi_{2}e^{\chi h^{-\rho}\dist(\x,\mathcal{M})}\psi_{h}\|^2_{\sL^2(\Omega)}\\
\nonumber&\geq \left((1-Ch^{\frac{p-2}{6p}})\lambda(\mathcal{G}, h,2)-\Tilde Ch^{2-2\rho}\right)\|\chi_{2}e^{\chi h^{-\rho}\dist(\x,\mathcal{M})}\psi_{h}\|^2_{\sL^2(\Omega)}\\
\nonumber&\geq ch\|\chi_{2}e^{\chi h^{-\rho}\dist(\x,\mathcal{M})}\psi_{h}\|^2_{\sL^2(\Omega)}\,.
\end{align}
where the constant $c>0$ comes from Proposition~\ref{prop.minorationL2}.

On the other hand, by support considerations, we have
\[\mathfrak{Q}^{\mathsf{NL}}_{\mathcal{G}, h}(\chi_{1}e^{\chi h^{-\rho}\dist(\x,\mathcal{M})}\psi_{h})=\mathfrak{Q}^{\mathsf{NL}}_{\mathcal{G}, h}(\chi_{1}\psi_{h})\,.\]
But we notice that
\[\mathfrak{Q}^{\mathsf{NL}}_{\mathcal{G}, h}(\chi_{1}\psi_{h})\geq \mathfrak{Q}_{\mathcal{G}, h}(\chi_{1}\psi_{h})-\lambda(\mathcal{G}, h, p)\int_{\Omega}|\chi_{1}\psi_{h}|^2|\psi_{h}|^{p-2}\dx\x\,.\]
Since $\psi_{h}$ is $\sL^p$-normalized and thanks to the H\"older inequality, we have
\[\int_{\Omega}|\chi_{1}\psi_{h}|^2|\psi_{h}|^{p-2}\dx\x\leq \|\chi_{1}\psi_{h}\|^2_{\sL^p(\Omega)}\,.\]
Thus we have
\begin{equation}\label{eq.chi1}
\mathfrak{Q}^{\mathsf{NL}}_{\mathcal{G}, h}(\chi_{1}e^{\chi h^{-\rho}\dist(\x,\mathcal{M})}\psi_{h})\geq 0\,.
\end{equation}
Combining \eqref{eq.loc.exp}, \eqref{eq.chi2}, \eqref{eq.chi1} and again that $\chi=0$ on the support of $\chi_{1}$, we get
\[ch\|\chi_{2}e^{\chi h^{-\rho}\dist(\x,\mathcal{M})}\psi_{h}\|^2_{\sL^2(\Omega)}\leq\Tilde C h^{2-2\rho}\|\psi_{h}\|^2_{\sL^2(\Omega)}\,.\]
From this last estimate (and using again that $\rho\in\left(0,\frac{1}{2}\right)$), it follows that
\[\|e^{\chi h^{-\rho}\dist(\x,\mathcal{M})}\psi_{h}\|^2_{\sL^2(\Omega)}\leq C\|\psi_{h}\|^2_{\sL^2(\Omega)}\,.\]
Now, we come back with this information to \eqref{eq.loc.exp} to deduce that
\[\mathfrak{Q}^{\mathsf{NL}}_{\mathcal{G}, h}(\chi_{2}e^{\chi h^{-\rho}\dist(\x,\mathcal{M})}\psi_{h})\leq C\|\psi_{h}\|^2_{\sL^2(\Omega)}\,,\]
and then, from \eqref{eq.controlNL}, we get
\[\mathfrak{Q}_{\mathcal{G}, h}(\chi_{2}e^{\chi h^{-\rho}\dist(\x,\mathcal{M})}\psi_{h})\leq Ch^{2-2\rho}\|\psi_{h}\|^2_{\sL^2(\Omega)}\leq \Tilde Ch^{2-2\rho}\|\psi_{h}\|^2_{\sL^p(\Omega)}\,,.\]
From the rough estimate $\mathfrak{Q}_{\mathcal{G}, h}(\psi_{h})=\lambda(\mathcal{G}, h, p)\leq C$ and the localization formula, we get
\[\mathfrak{Q}_{\mathcal{G}, h}(\chi_{1}e^{\chi h^{-\rho}\dist(\x,\mathcal{M})}\psi_{h})=\mathfrak{Q}_{\mathcal{G}, h}(\chi_{1}\psi_{h})\leq C\|\psi_{h}\|^2_{\sL^p(\Omega)}\,.\]
By using again the localization formula, we have
\[\mathfrak{Q}_{\mathcal{G}, h}(e^{\chi h^{-\rho}\dist(\x,\mathcal{M})}\psi_{h})\leq C\|\psi_{h}\|^2_{\sL^p(\Omega)}\,,\]
and thus
\[\|e^{\chi h^{-\rho}\dist(\x,\mathcal{M})}\psi_{h}\|_{\sL^p(\Omega)}\leq Ch^{-\gamma}\|\psi\|_{\sL^p(\Omega)}\,,\]
for some $\gamma>0$. The conclusion easily follows.
\end{proof}

\section{Continuity estimates}\label{sec.6}
In this section, we discuss different kinds of continuity results.
\subsection{Continuity}
The following proposition (jointly with Propositions~\ref{prop:interior} and~\ref{prop:upperbound}) implies the upper bound announced in Theorem~\ref{theo.1}.
\begin{proposition}\label{prop.cont}
Here we consider a constant geometry $\mathcal{G}_{\x}$ with $\x\in\overline{\Omega}$ and let $p\in(2,2^*)$. We have the following continuity properties.
\begin{enumerate}[\rm (i)]
\item\label{prop.conti} The function $\x\mapsto \lambda(\mathcal{G}_{\x}, 1, p)$ is continuous on $\Omega$.
\item\label{prop.contii} The function $\x\mapsto \lambda(\mathcal{G}_{\x}, 1, p)$ is continuous on $\partial\Omega$.
\item\label{prop.contiii} The function $\x\mapsto \lambda(\mathcal{G}_{\x}, 1, p)$ is lower semi-continuous on $\overline{\Omega}$.
\end{enumerate}
\end{proposition}
\begin{proof}
Let us consider \eqref{prop.conti}. This is a consequence of the results in \cite{BNvS,dCvS15}. 
Actually the reader may also adapt the forthcoming proof of \eqref{prop.contii} to get the continuity.

Let us now prove \eqref{prop.contii}. Without loss of generality, we may assume that $\Omega=\R^d_{+}$. Indeed we can simply use a rotation (smooth with respect to $\x$) and a change of gauge to rotate the problem from the affine tangent space to the tangent vector space. 

Let us consider a point $\x_{*}\in\partial\Omega$ and a sequence $(\x_{n})_{n\in\mathbb{N}}\subset \partial \Omega$ such that $\x_{n}\to \x_{*}$ when $n$ goes to infinity. The proof is divided into two steps. First, it is rather easy to get
\[\limsup_{n\to+\infty} \lambda(\mathcal{G}_{\x_{n}}, 1, p)\leq \lambda(\mathcal{G}_{\x_{*}}, 1, p)\,,\qquad \text{ for } \x_{n}\to \x_{*}\in\partial\Omega\,,\]
by using a minimizer associated with $\lambda(\mathcal{G}_{\x_{*}}, 1, p)$ as test function. Then, we shall establish
\[\liminf_{n\to+\infty} \lambda(\mathcal{G}_{\x_{n}}, 1, p)\geq \lambda(\mathcal{G}_{\x_{*}}, 1, p)\,,\qquad \text{ for } \x_{n}\to \x_{*}\in\partial\Omega\,.\]
This last inequality is slightly more subtle and uses the concentration-compactness strategy.

To simplify the notation, we denote $\mathfrak{Q}_{\mathcal{G}_{\x_{n}},1}=\mathfrak{Q}_{n}$ and
\[\begin{split}
  &\lambda(\mathcal{G}_{\x_{n}}, 1, p)=:\lambda_{n}^+\,,\quad\;\liminf_{n\to+\infty}\lambda_{n}^+=:\lambda_{*}^+\,,\\ 
  &\lambda(\underline{\mathcal{G}_{\x_{n}}}, 1, p)=:\lambda_{n}\,,\quad\;\liminf_{n\to+\infty}\lambda_{n}=:\lambda_{*}\,.
\end{split}\]
Let $\psi_{n}$ be an $\sL^p$-normalized function such that. 
\[
  \mathfrak{Q}_{n}(\psi_{n})= \eps_n + \lambda_{n}^+\,
\]
with $\lim_{n\to+\infty}\eps_n = 0$.
We notice that by the positivity property \eqref{a.0ii} in Assumption~\ref{a.0}, $(\psi_{n})$ is bounded in $\sL^2(\R^d_{+})$ and $\sH^1_{\mathsf{loc}}(\R^d_{+})$. By diamagnetism and Sobolev embedding, we infer that $(|\psi_{n}|)$ is also bounded in $\sH^{1}(\R^d_{+})$. 
We are left with concentration-compactness type arguments.
\begin{enumerate}[(a)]
\item Let us deal with the boundary vanishing.

Assume first that for all $R>0$
\[
  \lim_{n\to +\infty} M_R(\psi_n) = 0.
\]
By Proposition~\ref{prop.exclude-van}, we have that for all $R>0$
\[
  \lim_{n\to +\infty} \|\psi_{n}\|_{\sL^p(\Sigma_{R})} = 0
\]
Let us consider a quadratic partition of the unity
\[\chi^2_{R,1}+\chi^2_{R,2}=1\,,\]
with $\supp\chi_{R,1}$ supported in a neighborhood of the boundary of size $R$ and such that \[\Vert \nabla \chi_{R, 1}\Vert_{\sL^\infty} \le \frac{C}{R}.\] 
For  $n\geq 1$, we have, by the usual localization formula (and the fact that $(\psi_{n})$ is bounded in $\sL^2$),
\[
\mathfrak{Q}_{n}(\psi_{n})\geq\mathfrak{Q}_{n}(\Tilde\chi_{R,1}\psi_{n})+\mathfrak{Q}_{n}(\Tilde\chi_{R,2}\psi_{n})-\frac{C}{R^2}
\geq \mathfrak{Q}_{n}(\Tilde\chi_{R,2}\psi_{n}) - \frac{C}{R^2}\,.
\]
Therefore, it follows that
\[
   \lambda_n^+\geq   \lambda_n\|\chi_{R,2}\psi_{n}\|^2_{\sL^p(\R^d)}-\frac{C}{R^2}-\eps_n\,.
\]
We take the limit $n\to+\infty$ and then $R\to+\infty$, 
\[
\liminf_{n\to+\infty}\lambda_{n}^+\geq   \lim_{n\to+\infty}\lambda_n =  \lambda_*\geq \lambda_*^+.
\]
This is the result that we want to prove.

\item Let us now exclude the dichotomy. We consider the case when there exists $R_{0}>0$ such that $M_{R_{0}}(\psi_{n})$ does not go to $0$. 
Up to extraction and magnetic translations parallel to \(\partial \R^d_+\), we may assume that $(\psi_{n})_{n \ge 1}$ weakly converges in  $\sH^1_{\mathsf{loc}}(\R^d_{+})$ and in $\sL^q(\R^d_{+})$  to some $\psi_{*} \ne 0$ for all $q\in[2, 2^*]$.

Assume by contradiction that $\alpha := \|\psi_*\|_{\sL^p(\R^d_{+})}^p<1$.
Let us again consider a quadratic partition of the unity
\[\Tilde\chi^2_{R,1}+\Tilde\chi^2_{R,2}=1\,,\]
with $\supp\Tilde\chi_{R,1}\subset D(0,R)$ such that \(\Vert \nabla \chi_{R, 1}\Vert_{\sL^\infty} \le \frac{C}{R}\). 
For all $R\geq R_{0}$ and $n\geq 1$, as previously, we have, 
\[
\mathfrak{Q}_{n}(\psi_{n})\geq\mathfrak{Q}_{n}(\Tilde\chi_{R,1}\psi_{n})+\mathfrak{Q}_{n}(\Tilde\chi_{R,2}\psi_{n})-\frac{C}{R^2}-\eps_n\,.
\]
In particular, we get 
\[\begin{split}
  \lambda_{n}^+&\geq \lambda_{n}^+\left(\|\Tilde\chi_{R,1}\psi_{n}\|_{\sL^p(\R^d_{+})}^2+ \|\Tilde\chi_{R,2}\psi_{n}\|^2_{\sL^p(\R^d_{+})}\right)-\frac{C}{R^2}-\eps_n\,\\
  &\geq  \lambda_{n}^+\left(\|\psi_{n}\|_{\sL^p(B_R\cap\R^d_{+}))}^2+ (1-\|\psi_{n}\|^p_{\sL^p(B_{2R}\cap\R^d_{+})})^{\frac{2}{p}}\right)-\frac{C}{R^2}-\eps_n\,.
\end{split}\]
Taking the limit $n\to+\infty$ and then $R\to+\infty$, we obtain
\[
   \liminf_{n\to+\infty}\lambda_{n}^+\geq   \liminf_{n\to+\infty}\lambda_{n}^+\left(\alpha^{\frac{2}{p}}+ (1-\alpha)^{\frac{2}{p}}\right).
\]
This contradicts the concavity of $\alpha\mapsto \alpha^{\frac{2}{p}}$ so that $\|\psi_*\|_{\sL^p(\R^d_{+})}^p=1$

\item Finally we consider the pre-compact case. We obtain then that $(\psi_n)$ converges strongly to $\psi_*$ in $\sL^p(\R^d_{+})$. We also start with the localization formula
\[
\mathfrak{Q}_{n}(\psi_{n})\geq\mathfrak{Q}_{n}(\Tilde\chi_{R,1}\psi_{n})-\frac{C}{R^2}\,.
\]
By the weak convergence in \(\sH^1_{\mathrm{loc}}(\R^d_+)\), we obtain for each \(R > 0\)
\[
 \liminf_{n \to \infty}\lambda_n^+
\geq \mathfrak{Q}_{*}(\Tilde\chi_{R,1}\psi_*)\geq\lambda^+_*\|\Tilde\chi_{R,1}\psi_{*}\|^2_{\sL^p(\R^d_{+})}-\frac{C}{R^2}\,.
\]
Taking then the limit $R\to+\infty$, we get:
\[
 \liminf_{n \to \infty}\lambda_n^+
\geq \lambda^+_*.
\]
This is the result that we want to prove.
\end{enumerate}
Finally, concerning \eqref{prop.contiii}, it is sufficient to combine \eqref{prop.conti} and \eqref{prop.contii} with the fact that 
\[
  \forall \x\in\partial\Omega\,,\quad\lambda(\mathcal{G}_{\x}, 1, p)
  \leq \lambda(\underline{\mathcal{G}_{\x}}, 1, p)\,.\qedhere
\]
\end{proof}

\subsection{Sufficient conditions}\label{sec.suff}
In this section, we discuss some sufficient conditions which ensure that Assumption~\ref{a.0} can be satisfied. These conditions are based on the following non-degeneracy result.
\begin{proposition}\label{prop.Neumann}
The bottom of the spectrum $\lambda(\B)$ of the Neumann Laplacian with constant magnetic field on $\sL^2(\R^d_{+})$ satisfies
\[\max\left(\Theta_{0}\|\B^{\parallel}\|_{2}, \Tr^+\B^\perp\right)\leq\lambda(\B)\,,\]
where
\[\B=\left(\begin{array}{c | c}
\B^\perp&\B^\parallel\\
\hline
-(\B^\parallel)^{\mathsf{T}}&0\\
\end{array}\right)\,,\]
 and $\B^\parallel$ is a vector belonging to $\R^{d-1}$ and $\B^\perp$ is an skew-symmetric matrix of size $d-1$. The constant $\Theta_{0}\in(0,1)$ (sometimes called de Gennes constant, see for instance \cite[Chapter 4]{FouHel10}) is the bottom of the spectrum of the Neumann magnetic Laplacian on $\R^2_{+}$ when the magnetic field is constant equal to $1$.
\end{proposition}
\begin{proof}
Let us consider the rotations that preserve the $x_{d}$-axis. They are in the form
\[\mathsf{Q}=\left(\begin{array}{c | c}
Q&0\\
\hline
0&1\\
\end{array}\right)\,,\]
with $Q\in SO(d-1)$. Letting $\x=\mathsf{Q}\y$, $\mathfrak{L}_{\A}$ is unitarily equivalent to the following operator acting on $\sL^2(\R^d_{+})$:
\[\left(-i\nabla_{\y}+\Tilde\A(\y)\right)^2\,,\qquad\text{ where }\quad \Tilde\A(\y)=\frac{1}{2}\Tilde{\B}\y\,,\quad \Tilde{\B}=\mathsf{Q}^\mathsf{T}\B \mathsf{Q}\,.\]
It is clear that we may find a rotation that sends $\B^\parallel$ onto $\|\B^\parallel\|_{2}\e_{d-1}$ and we may assume that $B_{k d}=0$ for $1\leq k\leq d-2$. We notice that
\begin{equation}\label{eq.tildeA0}
\Tilde\A_{0}(\y^{\parallel})=\Tilde\A(\y^\parallel,0)=\frac{1}{2}Q^\mathsf{T}\B^\perp Q\y^\parallel\,.
\end{equation}
The magnetic Laplacian is now in the form
\[\left(-i\nabla_{\y}+\Tilde\A(\y)\right)^2=(D_{y_{d}}-\frac{1}{2}\Tilde B_{d\, d-1}y_{d-1})^2+(D_{y_{d-1}}+\Tilde A_{d-1})^2+\sum_{\ell=1}^{d-2}(D_{y_{\ell}}+\Tilde A_{\ell})^2\,,\]
with
\[\Tilde A_{d-1}=\frac{1}{2}\Tilde B_{d\,d-1}y_{d}+\hat A_{d-1}\,,\]
where $\hat A_{d-1}$ and $(\Tilde A_{\ell})_{1\leq \ell\leq d-2}$ are independent from $y_{d}$. After a change of gauge, we may consider the equivalent operator
\[\mathfrak{L}=D_{y_{d}}^2+\left(D_{y_{d-1}}+\Tilde B_{d\,d-1}y_{d}+\sum_{\ell=1}^{d-2} \Tilde B_{\ell d-1}y_{\ell}\right)^2+\sum_{\ell=1}^{d-2}(D_{y_{\ell}}+\Tilde A_{\ell}(y_{1},\ldots, y_{d-2},0, 0))^2\,.\]
Therefore, there is no more dependence on $y_{d-1}$. We may notice that, by definition, the lower dimension magnetic Laplacian
\[\mathfrak{L}^\perp=\left(D_{y_{d-1}}+\sum_{\ell=1}^{d-2} \Tilde B_{\ell d-1}y_{\ell}\right)^2+\sum_{\ell=1}^{d-2}(D_{y_{\ell}}+\Tilde A_{\ell}(y_{1},\ldots, y_{d-2},0, 0))^2\]
admits $\B^{\perp}$ as magnetic matrix.

Then, we notice that, for all $\psi$ in the domain of $\mathfrak{L}$, we have
\begin{multline*}
\langle\mathfrak{L}\psi,\psi\rangle_{\sL^2(\R^d_{+})}\geq\\
\int_{(y_{1},\ldots, y_{d-2})\in\R^{d-2}} \int_{(y_{d-1}, y_{d})\in\R^2_{+}} |D_{y_{d}}\psi|^2+|(D_{y_{d-1}}+\Tilde B_{d\,d-1}y_{d}+\sum_{\ell=1}^{d-2} \Tilde B_{\ell d-1}y_{\ell})\psi|^2\dx y_{d-1}\dx y_{d}\,.
\end{multline*}
Thus, by using a partial change of gauge (for fixed $(y_{1},\ldots, y_{d-2})$) to eliminate the sum term, we infer that
\[\langle\mathfrak{L}\psi,\psi\rangle_{\sL^2(\R^d_{+})}\geq \Theta_{0}|\Tilde B_{d\,d-1}|\|\psi\|^2_{\sL^2(\R^d_{+})}=\Theta_{0}\|\B^\parallel\|_{2}\|\psi\|^2_{\sL^2(\R^d_{+})}\,.\]
In the same way, dropping the term in $D_{y_{d}}$ and canceling the term $\Tilde B_{d\,d-1}y_{d}$ by a partial change of gauge, we find
\[\langle\mathfrak{L}\psi,\psi\rangle_{\sL^2(\R^d_{+})}\geq\Tr^+\B^\perp\|\psi\|^2_{\sL^2(\R^d_{+})}\,.\]
The conclusion follows by the min-max principle.
\end{proof}
We can now prove the following proposition.
\begin{proposition}
We have the following sufficient conditions.
\begin{enumerate}[i.]
\item Point \eqref{a.0i} in Assumption~\ref{a.0} is satisfied when $V\geq 0$ and $\B$ does not vanish on $\Omega$.
\item Point \eqref{a.0ii} in  Assumption~\ref{a.0} is satisfied when
\[\displaystyle{\inf_{\x\in\partial\Omega}\lambda((\x+\mathsf{T}_{\x}(\partial\Omega)+\R_{+}\n(\x),\mathsf{Id}, V(\x), \mathcal{A}^\mathsf{L}_{\x},0),1,2)>0}\] 
and if $\displaystyle{\max_{\partial\Omega}\gamma_-(\x)}$ is small enough where $\gamma_-(\x) = \max(0,-\gamma(\x))$.
\item Point \eqref{a.0ii} in  Assumption~\ref{a.0} is satisfied when $V\geq 0$, $\B$ does not vanish on the boundary and when $\displaystyle{\max_{\partial\Omega}\gamma_-(\x)}$ is small enough.
\end{enumerate}
\end{proposition}
\begin{proof}
The first point is obvious since $\Tr^+\,\B(\x)=0$ implies that $\B(\x)=0$. 

Let us now consider the second point. For $\x\in\partial\Omega$, we have, by the min-max principle (and splitting the electro-magnetic energy into two equal parts), for all $\psi\in\Dom(\mathfrak{Q}_{\mathcal{G}_{\x}, 1})$,
\[\mathfrak{Q}_{\mathcal{G}_{\x}, 1}(\psi)\geq \frac{c}{2}\|\psi\|^2_{\sL^2(\mathsf{T}_{\x})}+\frac{1}{2}\int_{\mathsf{T}_{\x}} |(-i\nabla+ \mathcal{A}^\mathsf{L}_{\x})\psi|^2+V(\x)|\psi|^2\dx\y-m\int_{\partial\mathsf{T}_{\x}}|\psi|^2\dx\sigma\,,\]
where $\mathsf{T}_{\x}=\x+\mathsf{T}_{\x}(\partial\Omega)+\R_{+}\n(\x)$ and
\[c=\inf_{\x\in\partial\Omega}\lambda(\mathsf{T}_{\x},\mathsf{Id}, V(\x), \mathcal{A}^\mathsf{L}_{\x},0),1,2)\,,\quad m=\max_{\x\in\partial\Omega}\gamma_{-}(\x)\,.\]
Then, by using the diamagnetic inequality and Lemma~\ref{lem.trace-eps}, we have
\[\int_{\partial\mathsf{T}_{\x}}|\psi|^2\dx\sigma\leq \eps \int_{\mathsf{T}_{\x}} |(-i\nabla+ \mathcal{A}^\mathsf{L}_{\x})\psi|^2\dx\y+C\eps^{-1}\|\psi\|^2_{\sL^2(\mathsf{T}_{\x})}\,.\]
Let us introduce $\displaystyle{M=\max_{\x\in\overline{\Omega}} |V(\x)}|$. It follows that
\[\int_{\partial\mathsf{T}_{\x}}|\psi|^2\dx\sigma\leq \eps \int_{\mathsf{T}_{\x}} |(-i\nabla+ \mathcal{A}^\mathsf{L}_{\x})\psi|^2+V(\x)|\psi|^2\dx\y+(\eps M+C\eps^{-1})\|\psi\|^2_{\sL^2(\mathsf{T}_{\x})}\,,\]
and thus, for $\eps\in\left(0,\frac{1}{2m}\right)$,
\[\mathfrak{Q}_{\mathcal{G}_{\x}, 1}(\psi)\geq \left(\frac{c}{2}-Mm\eps-mC\eps^{-1}\right)\|\psi\|^2_{\sL^2(\mathsf{T}_{\x})}\,.\]
We choose $\eps=\min\left(\frac{c}{4mM},\frac{1}{2m}\right)$ and then $m$ small enough to get
\[\mathfrak{Q}_{\mathcal{G}_{\x}, 1}(\psi)\geq \frac{c}{8}\|\psi\|^2_{\sL^2(\mathsf{T}_{\x})}\,.\]
The conclusion follows.

To get the third assertion, we notice that, for all $\x\in\partial\Omega$,
\[\lambda((\mathsf{T}_{\x},\mathsf{Id}, V(\x), \mathcal{A}^\mathsf{L}_{\x},0),1,2)\geq \lambda(\B_{\x})\geq \max\left(\Theta_{0}\|\B_{\x}^{\parallel}\|_{2}, \Tr^+\B_{\x}^\perp\right)\,,\]
where we have used Proposition~\ref{prop.Neumann}. Then, the lower bound is a continuous and positive function of $\x$ on the compact set $\partial\Omega$ and we may apply the result of the second point.
\end{proof}

\subsection{The Dirichlet case}
In this section we discuss the existence of the minimizers when $\Omega=\R^d_{+}$, $V=0$, $\B$ is uniform, non-zero and when the boundary carries the Dirichlet condition. 

\begin{proposition}
Here $d\geq 3$. Let us consider $\lambda=\lambda((\R^d_{+}, \mathsf{Id}, 0, \A,+\infty), p, 1)$ with $\A\in\mathcal{L}(\R^d)$ such that $\dx\A$ is not zero. Then $\lambda$ is not attained. Moreover, if we let $\underline{\lambda}=\lambda((\R^d, \mathsf{Id}, 0, \A,0), p, 1)$, then $\lambda=\underline{\lambda}$.
\end{proposition}

\begin{proof}
We recall that we always have $\lambda\leq\underline{\lambda}$. Indeed $\underline{\lambda}$ is a minimum and any associated minimizer has an exponential decay: it is sufficient to translate any minimizer to infinity and use a cutoff function.

We next claim that \(\lambda \ge \underline{\lambda}\). Indeed, if $\varphi$ is a test function for the problem in \(\R^d_+\), we extend $\varphi$ by zero and denote by $\underline{\varphi}\in\sH^1(\R^d)$ its extension. We use $\underline{\varphi}$ as test function for the quadratic form on $\R^d$ and get $\underline{\lambda}\leq\lambda$. 

Let us assume that $\lambda$ is attained for a function $\psi\in\sH^1_{0}(\Omega)$ with $\|\psi\|_{\sL^p(\Omega)}=1$. Let $\underline{\psi}\in\sH^1(\R^d)$ be the extension of $\psi$ by $0$. Therefore $\underline{\psi}$ is a minimizer associated with $\underline{\lambda}$ and it vanishes on a non-empty open set. It also satisfies the elliptic equation (the associated Euler-Lagrange equation):
\[(-i\nabla+\A)^2\underline{\psi}=\underline{\lambda}|\underline{\psi}|^{p-2}\underline{\psi}\,.\]
By Sobolev embedding, we have $\underline{\psi}\in\sL^{2^*}(\R^d)$, with $2^*=\frac{2d}{d-2}$. Let us consider any bounded open set $U\subset \R^d$. If $0<\frac{d}{2}(p-2)\leq 2^*$, we have, by the H\"older inequality, $|\underline{\psi}|^{p-2}\in\sL^{\frac{d}{2}}(U)$. But $\frac{d}{2}(p-2)\leq 2^*$ is equivalent to $p\leq 2+\frac{4}{d-2}=2^*$. Thus we have 
$|\underline{\psi}|^{p-2}\in\sL_{\loc}^{\frac{d}{2}}(\R^d)$. From this and the fact that $\underline{\psi}\in\sL^{2^*}(\R^d)$ we get, with the H\"older inequality, $\underline{\psi}\in\sL^2_{\loc}(\R^d)$. We infer that $\underline{\psi}\in\sH^2_{\loc}(\R^d)$. The assumptions of \cite[Theorem 1.1]{Sogge89} are satisfied (since $\frac{2d}{d+2}\leq 2$) so that the unique continuation property holds for $\underline{\psi}$. We deduce that $\underline{\psi}=0$ and this is a contradiction.
\end{proof}

\section{Bidimensional waveguides}\label{sec.7}
This section is devoted to the proof of Proposition~\ref{prop.waveguides}.

\subsection{Reduction to the straight waveguide}
Let us first pull back the variable geometry onto the homogeneous geometry. 

\begin{lemma}
There exist $h_{0}, C>0$ such that for all $h\in(0,h_{0})$ and for all $\psi\in\sH^1_{0}(\Sigma_{h})$, we have the comparison
\[(1-Ch)h^{-1-\frac{2}{p}} \frac{\mathfrak{Q}_{\Sigma, a, h}(\varphi)}{\|\varphi\|^2_{\sL^p(\Sigma)}}\leq\frac{\int_{\Sigma_{h}}|\nabla\psi|^2\dx \x}{\|\psi\|^2_{\sL^p(\Sigma_{h})}}\leq (1+Ch)h^{-1-\frac{2}{p}} \frac{\mathfrak{Q}_{\Sigma, a, h}(\varphi)}{\|\varphi\|^2_{\sL^p(\Sigma)}},\]
where $\varphi(s, t)=a(s)^{\frac{1}{p}}\psi(\Phi_{h}(s, t))$ and
\[\mathfrak{Q}_{\Sigma, a, h}(\varphi)=\int_{\Sigma} h^2a^{1-\frac{2}{p}}|\partial_{s}\varphi|^2+a^{-1-\frac{2}{p}}|\partial_{t}\varphi|^2\dx s\dx t\,.\]
\end{lemma}

\begin{proof}
We notice that
\[d\Phi_{h}=[(1-tkh a)\Gamma'+th a'\n, h a \n]\,.\]
so that
\[G_{h}=(d\Phi_{h})^{\mathsf{T}}d\Phi_{h}=
\begin{pmatrix}
(1-tkh a)^2+h^2 a^2 t^2&th^2 a a'\\
th^2 a a'&h^2 a^2
\end{pmatrix}\,.\]
We get that $|G_{h}|=\det G_{h}=a^2h^2(1+\mathcal{O}(h))$. In the sense of quadratic forms, we have
\[G^{-1}_{h}=(1+\mathcal{O}(h))\begin{pmatrix}
1&0\\
0&h^{-2} a^{-2}
\end{pmatrix}\,.\]
Thanks to the change of variables $\x=\Phi_{h}(s, t)$, we deduce the following comparison:
\[(1-Ch)h^{-1}\widetilde{\mathfrak{Q}}_{\Sigma, a, h}(\Tilde\psi)\leq\int_{\Sigma_{h}}|\nabla\psi|^2\dx \x\leq (1+Ch)h^{-1}\widetilde{\mathfrak{Q}}_{\Sigma, a, h}(\Tilde\psi)\,,\]
where
\[\widetilde{\mathfrak{Q}}_{\Sigma, a, h}(\Tilde\psi)=\int_{\Sigma} ah^2|\partial_{s}\Tilde\psi|^2+a^{-1}|\partial_{t}\Tilde\psi|^2\dx s\dx t\,.\]
In the same way, we get
\[(1-Ch)h^{\frac{2}{p}}\left(\int_{\Sigma} |\Tilde\psi|^p a\dx s \dx t\right)^{\frac{2}{p}}\leq\left(\int_{\Sigma} |\psi|^p\dx\x\right)^{\frac{2}{p}}\leq(1+Ch)h^{\frac{2}{p}}\left(\int_{\Sigma} |\Tilde\psi|^p a\dx s \dx t\right)^{\frac{2}{p}}\,.\]
We introduce the change of function $\Tilde\psi= a^{-\frac{1}{p}}\varphi$ so that $\int_{\Sigma} |\Tilde\psi|^p a\dx s \dx t=\int_{\Sigma} |\varphi|^p \dx s \dx t$.
By a computation and an integration by parts, it follows that
\[\widetilde{\mathfrak{Q}}_{\Sigma, a, h}(\Tilde\psi)=\int_{\Sigma} h^2a^{1-\frac{2}{p}}|\partial_{s}\varphi|^2+h^2V(s)|\varphi|^2+a^{-1-\frac{2}{p}}|\partial_{t}\varphi|^2\dx s\dx t\,,\]
with 
\[V(s)=\frac{1}{p^2}a'^2a^{-\frac{2}{p}-1}+\frac{1}{p}\partial_{s}\left(a'a^{-\frac{2}{p}}\right)\,.\]
Note that there exists $c>0$ such that, for all $h>0$,
\[\widetilde{\mathfrak{Q}}_{\Sigma, a, h}(\Tilde\psi)\geq c\|\varphi\|^2_{\sL^2(\Sigma)}\,,\]
and that $V\in\sL^\infty(\R)$. We get
\[(1-Ch)\mathfrak{Q}_{\Sigma, a, h}(\varphi)\leq\widetilde{\mathfrak{Q}}_{\Sigma, h}(\Tilde\psi)\leq(1+Ch)\mathfrak{Q}_{\Sigma, h}(\varphi)\,.\qedhere
\]
\end{proof}
Therefore we are reduced to consider the minimization problem
\[\text{ minimize}\quad \frac{\mathfrak{Q}_{\Sigma, a, h}(\varphi)}{\|\varphi\|^2_{\sL^p(\Sigma)}}\,,\text{ for }  \varphi\in\sH_{0}^1(\Sigma)\,.\]
\subsection{Estimate of the normalized Sobolev quotient}
We are now in position to use the strategy developed in this paper on $\mathfrak{Q}_{\Sigma, a, h}$. Note here that we have a partially semiclassical problem. First, we have to establish an upper bound for the Sobolev quotient. For that purpose, we must freeze the height $a$ to the maximal height $a_{\max}$, attained at some point $s_{\max}$. We will need the following lemma.
\begin{lemma}\label{lem.Sigma1}
The Sobolev constant $\lambda^\Dir(\Sigma, p)$ is attained. Any corresponding minimizer has an exponential decay.
\end{lemma}
\begin{proof}
Once the Sobolev is attained, it is rather clear that the minimizers have an exponential decay (see the proof of Proposition~\ref{prop.exp}). The fact that the infimum is attained is a consequence of a concentration-compactness investigation along the $s$-axis: we are in the compactness case modulo translations parallel to the $s$-axis.
\end{proof}

\begin{lemma}\label{lem.ub-wg}
There exist $h_{0}, C>0$ such that, for all $h\in(0, h_{0})$,
\[\inf_{\underset{\psi\neq 0}{\varphi\in \sH^1(\Sigma),}}\frac{\mathfrak{Q}_{\Sigma, a, h}(\varphi)}{\|\varphi\|^2_{\sL^p(\Sigma)}}\leq (1+Ch^2)h^{1-\frac{2}{p}} a^{-\frac{4}{p}}_{\max}\lambda^\Dir(\Sigma, p)\,.\]
\end{lemma}
\begin{proof}
Let us consider an   $\sL^p$-normalized minimizer $\phi_{0}$ associated with the $p$-eigenvalue $\lambda(\Sigma, \mathsf{Id}, 0, 0, +\infty, 1, p)$ and introduce
\[\varphi_{h}(s,t)=\phi_{0}\left(a^{-1}_{\max}\frac{s-s_{\max}}{h},t\right)\,.\]
We compute
\[\mathfrak{Q}_{\Sigma, a, h}(\varphi_{h})=h a_{\max}\int_{\Sigma} \left\{a^{-2}_{\max}a_{h}^{1-\frac{2}{p}}|\partial_{s}\phi_{0}\left(\sigma,t\right)|^2+a_{h}^{-1-\frac{2}{p}}|\partial_{t}\phi_{0}\left(\sigma,t\right)|^2\right\} \dx \sigma\dx t\,,\]
with $a_{h}(\sigma, t)=a\left(s_{\max}+h a_{\max} \sigma,t\right)$. Thanks to a Taylor expansion and to the exponential decay of $\phi_{0}$, we get
\[\mathfrak{Q}_{\Sigma, a, h}(\varphi_{h})\leq (1+Ch^2)h a_{\max}^{-\frac{2}{p}}\lambda(\Sigma, \mathsf{Id}, 0, 0, +\infty, 1, p)\,.\]
We also get
\[\left(\int_{\Sigma} \left|\varphi_{h}\right|^p\dx \sigma\dx t\right)^{\frac{2}{p}}=h^{\frac{2}{p}} a^{\frac{2}{p}}_{\max} \,.\]
The conclusion follows.
\end{proof}
Let us now deal with the lower bound.
\begin{lemma}\label{lem.lb-wg}
There exist $h_{0}, C>0$ such that, for all $h\in(0, h_{0})$ and all $\varphi\in\sH_{0}^1(\Sigma),$
\[\mathfrak{Q}_{\Sigma, a, h}(\varphi)\geq  (1-Ch^{\frac{1}{2}})h^{1-\frac{2}{p}} a^{-\frac{4}{p}}_{\max}\lambda^\Dir(\Sigma, p)\|\varphi\|^2_{\sL^p(\Sigma)}\,.\]
\end{lemma}
\begin{proof}
Let us use a \enquote{sliding} partition of the unity as in Section~\ref{sec.sliding} but only with respect to $s$ (\emph{i.e.} $d=1$). We recall \eqref{eq.part-quad} and \eqref{partition-remainder}. By using the partition adapted to $\varphi$, we have
\begin{equation*}
\mathfrak{Q}_{\Sigma, a, h}(\varphi)\geq\sum_{\kb\in\Z}\mathfrak{Q}_{\Sigma, a, h}(\Tilde\chi_{\alpha,\rho,h}^{[\kb]}\varphi) -\Tilde Dh^{2-\alpha-\rho}\|\varphi\|_{\sL^2(\Sigma)}^2\,,
\end{equation*}
so that
\begin{equation*}
\mathfrak{Q}_{\Sigma, a, h}(\varphi)\geq (1-Ch^{2-\alpha-\rho})\sum_{\kb\in\Z}\mathfrak{Q}_{\Sigma, a, h}(\Tilde\chi_{\alpha,\rho,h}^{[\kb]}\varphi)\,.
\end{equation*}
Then, by a support consideration and a Taylor expansion, we get
\begin{equation*}
\mathfrak{Q}_{\Sigma, a, h}(\varphi)\geq (1-Ch^{2-\alpha-\rho})(1-Ch^{\rho})\sum_{\kb\in\Z}\mathfrak{Q}_{\Sigma, a(s_{\kb}), h}(\Tilde\chi_{\alpha,\rho,h}^{[\kb]}\varphi)\,,
\end{equation*}
so that, by rescaling and a straightforward comparison,
\begin{equation*}
\mathfrak{Q}_{\Sigma, a, h}(\varphi)\geq \lambda^\Dir(\Sigma, p)h^{1-\frac{2}{p}}(1-Ch^{2-\alpha-\rho})(1-Ch^{\rho})\sum_{\kb\in\Z}a(s_{\kb})^{-\frac{4}{p}}\|\Tilde\chi_{\alpha,\rho,h}^{[\kb]}\varphi\|^2_{\sL^p(\Sigma)}\,.
\end{equation*}
Since $a(s_{\kb})^{-\frac{4}{p}}\geq a_{\max}^{-\frac{4}{p}}$ and by using that the partition is adapted to $\varphi$, we get
\begin{equation*}
\mathfrak{Q}_{\Sigma, a, h}(\varphi)\geq  a_{\max}^{-\frac{4}{p}}\lambda^\Dir(\Sigma, p)h^{1-\frac{2}{p}}(1-Ch^{2-\alpha-\rho})(1-Ch^{\rho})(1-Ch^{\alpha-\rho})\|\varphi\|^2_{\sL^p(\Sigma)}\,.
\end{equation*}
Optimizing the remainders, we find $\rho=\frac{1}{2}$ and $\alpha=1$.
\end{proof}
We leave the proof of the corresponding localization estimates to the reader since they follow from straightforward adaptations of the methods developed in this paper.

\section{Some perspectives}\label{sec.pers}
In this last section, we discuss some perspectives and open problems. There are many possible directions to extend our investigation and we only select a few of them in the next lines.

Firstly, it would be quite interesting to analyze the case of domains with corners. In the semiclassical regime, the strategy developed in \cite{BDP15} (for the case $p=2$ and Neumann condition; see also \cite{BP15} where the same strategy is used in the non-magnetic Robin case) could likely apply to get semiclassical upper bounds (as in Theorem~\ref{theo.1}). Nevertheless several modifications should be made (in particular about the considerations involving a separation of variables or the Fourier transform). The semi-continuity (see Proposition~\ref{prop.semicont}) in the Robin case and/or in dimension higher than three does not seem to be obvious. For $p=2$, it is only known for the Neumann case with pure magnetic field in two and three dimensions (see \cite{BDP15}). For $p>2$, one should perform a concentration-compactness investigation along the singular chains. For the lower bound, the adaptations should be easier (with a change of the localization scale near the conical singularities). Even in the case without magnetic field, it would be quite interesting to analyze the $p$-eigenvalue $\lambda(\mathsf{G}, 1, p)$ when $\mathsf{G}=(U,\Id,  1,0,\mathsf{c})$ and where $U$ is a dihedral. It seems that the question to know if $\lambda(\mathsf{G}, 1, p)$ is attained is open (and the answer should strongly depend on $\mathsf{c}$ as we guess from Proposition~\ref{prop:Robin}).

Secondly, the waveguide situation could be extended to general partially semiclassical problems. For instance, one could first consider a partially semiclassical and pure electric interaction in $\R^d$. Many inhomogeneous situations lead to this kind of limit (especially in the case with magnetic field as we see in \cite{BHR15}). In the waveguides framework, the description of curvature effects on the asymptotics of $p$-eigenvalues seems to be an open area (for $p=2$, it is known to play a role in the lower order terms). In the case of waveguides of \emph{uniform} width, we do not even know if the energy of the nonlinear groundstate is strictly less than the nonlinear energy at infinity. These curvature effects on the $p$-eigenvalues in magnetic/Robin situations would be interesting as well, especially if we imagine that the non-linearity ($p>2$) amplifies the localization properties of the linear groundstates.

\subsection*{Acknowledgments}
This work was partially supported by the Henri Lebesgue Center (programme  \enquote{Investissements d'avenir}  -- n\textsuperscript{o} ANR-11-LABX-0020-01)
and by the Projet de Recherche FNRS T.1110.14 \enquote{Asymptotic properties of semilinear systems}.
S.F. was partially supported by a Sapere Aude grant from the Danish Councils for Independent Research, Grant number DFF--4181-00221.


\end{document}